\documentclass{amsart}
\usepackage{amsmath}
\usepackage{amsfonts}
\usepackage{amsthm}
\usepackage{amssymb}
\usepackage{cite}
\usepackage{enumerate}
\usepackage[all]{xy}
\usepackage{cite}
\usepackage{mathrsfs}
\usepackage{color}
\usepackage{xcolor}
\usepackage{hyperref}
\usepackage{fancyhdr}
\hypersetup{
	colorlinks=true,
	anchorcolor=blue,
	linkcolor=blue,
	filecolor=blue,
	urlcolor=blue,	
	citecolor=blue,
	bookmarks=true,
	bookmarksopen=true,
	pdfborder=000
}
\numberwithin{equation}{section}
\theoremstyle{plain}
\newtheorem{thm}{Theorem}[section]

\newtheorem{proposition}[thm]{Proposition}

\newtheorem{cor}[thm]{Corollary}

\newtheorem{pro}[thm]{Problem}
\newtheorem{lemma}[thm]{Lemma}
\newtheoremstyle{noparens}%
 {}{}%
 {\itshape}{}%
 {\bfseries}{.}%
 { }%
 {\thmname{#1}\thmnumber{ #2}\mdseries\thmnote{ #3}}
\theoremstyle{noparens}
\newtheorem{lemmaNoParens}[thm]{Lemma}
\newtheorem{thmNoParens}[thm]{Theorem}

\theoremstyle{definition}
\newtheorem{defn}[thm]{Definition}
\newtheorem{ex}[thm]{Example}
\theoremstyle{remark}
\newtheorem{rmk}[thm]{Remark}

\makeatletter

\newcommand{\Rmnum}[1]{\expandafter\@slowromancap\romannumeral #1@}

\newcommand{\K}{K\"{a}hler}
\newcommand{\T}{Teichm\"{u}ller}
\makeatother
\begin{document}
\title{Geometry of bounded domains with piecewise smooth boundary}

\author{Xingsi Pu\textsuperscript{1} $\&$ Lang Wang\textsuperscript{2} }
\address{$1.$ Mathematical Science Research Center, Chongqing University of Technology, Chongqing, 400054, China}
\address{$2.$ School of Mathematical Sciences, Guizhou Normal University, Guiyang, 550025, P.R. China.}
\email{puxs@cqut.edu.cn,\:wanglang2020@amss.ac.cn}


\begin{abstract}
In this paper, we study the geometry of bounded domains with piecewise smooth boundary. Specifically, we obtain the relationship between the squeezing function corresponding to polydisk and Levi flatness on bounded convex domains. As an application, we prove that a two-dimensional bounded generic convex  domain with piecewise $C^2$-smooth boundary that is not smooth and admits a  finite volume quotient is biholomorphic to bidisk. Moreover, we show that any \T\ space $\mathcal{T}_g$ with $g\geq2$ cannot be biholomorphic to a bounded domain with piecewise $C^2$-smooth boundary.
\end{abstract}

\maketitle
\section{Introduction}

In complex space $\mathbb{C}^n$,  it is a fundamental fact that every bounded domain with $C^2$-smooth boundary admits a strongly pseudoconvex boundary point. Strong pseudoconvexity yields numerous characterizations of the unit ball. A landmark result states that if a bounded strongly pseudoconvex domain has a non-compact automorphism group, then the domain is biholomorphic to the unit ball (one can see \cite{W77,R79} for more details).

On the other hand, Pinchuk \cite{pinchuk} introduced domains with piecewise $C^2$-smooth boundary and generalized this characterization to products of balls under the assumption of homogeneity. Moreover, Kodama \cite{K86} improved this result to the generic domains that have a compact quotient.
For generic convex domains with piecewise Levi flat boundary, Kim \cite{pf} characterized those possessing a non-compact automorphism group, and Fu-Wong \cite{fw} later extended this to two-dimensional simply connected generic domains with piecewise smooth Levi flat boundary. 

This paper is devoted to study the geometry of bounded domains with piecewise smooth boundary, a natural class that interpolates between smooth domains and product domains. For a smooth boundary point, strong pseudoconvexity forces the squeezing function to approach $1$. By duality, Levi flatness is expected to play an analogous role for the polydisk: it should force the squeezing function corresponding to polydisk to converge to 1 at boundary points of Levi flat hypersurfaces. Let $\Omega\subset\mathbb{C}^n$ be a bounded generic convex domain with piecewise Levi flat boundary, and $T_{\Omega}$ be the squeezing function corresponding to polydisk. Our first main result confirms this expected boundary behavior in the two-dimensional case.

\begin{thm}\label{main}
    For a two-dimensional bounded convex domain $\Omega$ with piecewise Levi flat boundary, if $p\in\partial\Omega$ is a generic boundary point, then we have
    \[
    \lim\limits_{z\rightarrow p}T_{\Omega}(z)=1.
    \]
\end{thm}

The proof of Theorem \ref{main} separates into the singular part and the smooth part (see Theorem \ref{ls} and Theorem \ref{sp1}). The main strategy is to rescale the domain so that it converges to a model domain, and then to construct a holomorphic embedding from the rescaled domain into the bidisk $\mathbb{D}^2$. 

Let $\Omega$ be the product of a disk and an annulus, then $\Omega$  is generic and has piecewise smooth Levi flat boundary and non-compact automorphism group. However, such $\Omega$ is not biholomorphic to bidisk, which means that Theorem \ref{main} fails in the setting of bounded generic domains with piecewise Levi flat boundary.
\begin{rmk}

       Let $\Omega\subset\mathbb{C}^n$ be a bounded convex domain with piecewise Levi flat boundary, if $p\in\partial\Omega$ is a generic singular point with index $n$ (see Section \ref{sec2}), the proof of Theorem \ref{ls} also yields that $\lim\limits_{z\rightarrow p}T_{\Omega}(z)=1$.

\end{rmk}

The proofs of Theorem \ref{sp} and Theorem \ref{sp1} yield the following example.

\begin{ex}
    Let $\mathbb{B}^2$ be the unit ball in $\mathbb{C}^2$, and $E=\left\{(z,w)\in\mathbb{C}^2:\operatorname{Im}z>0\right\}$. If $\Omega=\mathbb{B}^2\cap E$, then for each point $p\in\partial E\cap\mathbb{B}^2$, we have
    \[
    \lim_{\Omega\ni z\rightarrow p}T_{\Omega}(z)=1.
    \]
\end{ex}

\begin{rmk}$\hfill$
\begin{enumerate}
    \item For each $n\geq2$. Forn\ae ss and Wold \cite{forw} constructed a bounded convex domain $\Omega\subset\mathbb{C}^n$ with $C^2$-smooth boundary that is Levi flat at some boundary point $p$ and satisfies $\lim\limits_{z\rightarrow p}S_{\Omega}(z)=1$. Then Lemma \ref{gap} implies that 
   \[
   \limsup\limits_{z\rightarrow p}T_{\Omega}(z)<1
   \]
   for such $p$. This example shows that Levi flatness of a single smooth point cannot imply that the squeezing function corresponding to polydisk tends to one.

    \item Let $E:=\left\{(z,w)\in\mathbb{C}^2:|z|<1,|w|<1,|z+w|<\sqrt{2}\right\}$ and $\Omega=\mathbb{D}\times E\subset\mathbb{C}^3$. Then $\Omega$ is a bounded convex domain with piecewise Levi flat boundary. For $p=(1,0,0)$, it is a smooth boundary point and hence generic. However, for such $p$ we have
    \[
    \limsup_{z\to p}T_{\Omega}(z)<1,
    \]
    because $\Omega$ is not biholomorphic to $\mathbb{D}^3$. This examples demonstrates that Theorem \ref{main} fails in $\mathbb{C}^n$ with $n\geq3$.
\end{enumerate}
   
\end{rmk}

As an application of Theorem \ref{main},  we can give an alternative proof of \cite[Theorem 1]{pf}, which is stated as follows.
\begin{cor}\label{lfb}
 Let $\Omega\subset\mathbb{C}^2$ be a  bounded generic convex domain with piecewise Levi flat boundary. If $\operatorname{Aut}(\Omega)$ is non-compact, then $\Omega$ is biholomorphic to the bidisk in $\mathbb{C}^2$.
\end{cor}

  Furthermore, we characterize Levi flatness of smooth boundary points for domains that are biholomorphic to a two-dimensional bounded generic convex domain with piecewise Levi flat boundary.
\begin{thm}\label{elf}
    Let $\Omega$ be a two-dimensional bounded generic convex domain with piecewise Levi flat boundary, and $\Omega'$ be a bounded domain that is biholomorphic to $\Omega$.  Then  $\partial\Omega'$ is also Levi flat at every $C^2$-smooth point.
\end{thm}
In particular, the following result can be viewed as a dual counterpart of \cite[Question 4.5]{forw}.

\begin{cor}
    Suppose $\Omega$ is a two-dimensional bounded domain that is biholomorphic to bidisk $\mathbb{D}^2$, then $\partial\Omega$ is Levi flat at every $C^2$-smooth point.
\end{cor}

To prove Theorem \ref{elf}, we will establish a gap result relating the squeezing functions $S_{\Omega}$ and $T_{\Omega}$ (see Lemma \ref{gap}). Moreover, Lemma \ref{gap} can imply the following corollary.

\begin{cor}\label{flatness}
    Let $\Omega$ be a bounded convex domain in $\mathbb{C}^2$. If $\partial\Omega$ is locally $C^2$-smooth in a neighborhood of $p\in\partial\Omega$ and $\lim\limits_{z\rightarrow p}T_{\Omega}(z)=1$, then there exists a neighborhood $U$ of $p$ such that $\partial\Omega$ is Levi flat at every $q\in\partial\Omega\cap U$.
\end{cor}

For a bounded domain $\Omega\subset\mathbb{C}^n$ with $C^2$-smooth boundary, Zimmer \cite{finite} proved that: if $\Omega$ has a finite volume quotient with respect to either the Bergman volume, the \K-Einstein volume, or the Kobayashi-Eisenman volume, then it is biholomorphic to the unit ball. In the setting of bounded convex domains, a $folklore$ conjecture asserts that a bounded convex domain that admits a finite volume quotient is biholomorphic to a bounded symmetric space(see \cite{lw} for more details). 
Motivated by the main theorem of \cite{pinchuk2}, we can pose the following problem.
\begin{pro}
  Let $\Omega\subset\mathbb{C}^n$ be a bounded generic convex domain with piecewise $C^2$-smooth boundary and $\Gamma\leq\operatorname{Aut}(\Omega)$ be a discrete group acting freely on $\Omega$. Suppose $\Gamma\setminus\Omega$ has a finite volume, is $\Omega$ biholomorphic to $\mathbb{B}^{n_1}\times\mathbb{B}^{n_2}\times\cdots\times\mathbb{B}^{n_i}$ for some positive integers $n_1,\cdots,n_i$ satisfying $n_1+\cdots n_i=n$? 
\end{pro}

After applying Theorem \ref{main}, we prove the following two-dimensional case of this problem.

\begin{thm}\label{fv}
    Suppose $\Omega\subset\mathbb{C}^2$ is a bounded generic convex domain with piecewise $C^2$-smooth boundary that is not smooth, and $\Gamma\leq\operatorname{Aut}(\Omega)$ is a discrete group acting freely on $\Omega$. If $\Gamma\setminus\Omega$ has finite volume with respect to either Bergman volume, the K\"{a}hler-Einstein volume, or the Kobayashi-Eisenman volume, then $\Omega$ is biholomorphic to the bidisk. 
\end{thm}
\smallskip

We now turn to the characterization of Teichm\"{u}ller space within the class of bounded domains. Let $S$ be a compact surface with genus $g\geq2$, and $\mathcal{T}_g$ be the \T\ space associated to $S$. It's well known that any $\mathcal{T}_g$ cannot be biholomorphic to a bounded domain in $\mathbb{C}^n$ with $C^2$-smooth boundary. One can refer \cite{finite,tei} for more details. In the setting of bounded domains with piecewise $C^2$-smooth boundary, we can prove the following theorem.

\begin{thm}\label{te}
   Any Teichm\"{u}ller space $\mathcal{T}_g$ with $g\geq2$ cannot be biholomorphic to a bounded domain with piecewise $C^2$-smooth boundary.
\end{thm}

As an application of Theorem \ref{te}, we obtain the following.

\begin{cor}
    Any $\mathcal{T}_g$ with $g\geq2$ cannot be biholomorphic to a bounded domain with $C^2$-smooth boundary.
\end{cor}

For non-smooth bounded domain, the main theorem in \cite{ct} states that any \T\ space $\mathcal{T}_g$ with $g\geq2$ cannot be biholomorphic to a bounded convex domain. Later, Gupta and Seshadri \cite{tei} extended the result to the case of locally strictly convex domains. Moreover, they conjectured that local strict convexity can be relaxed to locally convexity.  Recently, Liu and Wang \cite{LW21} proved that such a domain cannot even be locally log-type convex near a boundary point.
In this paper, by relaxing locally strict convexity to  local convexifiability at a local holomorphic peak boundary point, we obtain the following Theorem \ref{peak}. Note that a boundary point that is locally strictly convex is automatically a local holomorphic peak point.

\begin{defn}
Let $\Omega$ be a domain in $\mathbb{C}^n$. A boundary point $p$ is said local\ holomorphic\  peak if there exists a neighborhood $U$ of $p$ and a holomorphic map $f:U\rightarrow\mathbb{C}$ such that $|f(z)|<1$ whenever $z\in U\cap\overline{\Omega}\setminus\{p\}$ and $f(p)=1$.  
\end{defn}

\begin{defn}\label{convexifiable}
    For a bounded domain $\Omega\subset\mathbb{C}^n$, a boundary point $p$ is said to be $locally\ convexifiable$ if there exists a neighborhood $U$ of $p$ and a holomorphic embedding $\Phi:U\to\mathbb{C}^n$ such that $\Phi(U\cap\Omega)$ is convex.
\end{defn}

\begin{thm}\label{peak}
    Any  Teichm\"uller space $\mathcal{T}_{g}$ with $g\geq2$ cannot be biholomorphic to a bounded domain  that is locally convexifiable at a local holomorphic peak point $p$.
\end{thm} 

After applying Theorem \ref{peak}, we can prove the following theorem that relaxes the locally strictly convex assumption in \cite{tei} to local convexifiability at every boundary point.
\begin{thm}\label{lcb}
    Any  Teichm\"uller space $\mathcal{T}_{g}$ with $g\geq2$ cannot be biholomorphic to a bounded domain  that is locally convexifiable at every boundary point.
\end{thm}

This paper is organized as follows. We give the preliminaries in Section \ref{sec2}. Section \ref{sec3} is devoted to introduce properties of space of convex domains. In Section \ref{sec4}, we study the geometry of bounded convex domains. The theorems about biholomorphism-type of \T\ space are proved in Section \ref{sec5}.

\section{Preliminaries}\label{sec2}
\subsection{Notations}

\begin{enumerate}
    \item $|\cdot|$ is denoted the standard Euclidean norm in $\mathbb{C}^n$.

    \item $\mathbb{D}$ is the unit disk in $\mathbb{C}$, $\mathbb{B}^n$ and $\mathbb{D}^n$ are the unit ball and unit polydisk in $\mathbb{C}^n$ for $n\geq2$ respectively.

    \item For each $n\geq1$, $B_r(z):=\{w\in\mathbb{C}^n:|z-w|<r\}$ and $D_r(z):=\{w\in\mathbb{C}^n:|z_i-w_i|<r,i=1,\cdots,n\}$ with $z=(z_1,\cdots,z_n)$.
    
    \item For a bounded pseudoconvex domain $\Omega$, let $k_{\Omega},g_B$ and $g_{KE}$ denote the Kobayashi metric, the Bergman metric and the \K-Einstein metric on $\Omega$ respectively.
\end{enumerate}

\subsection{Piecewise smooth boundary}$\hfill$

Following \cite{pinchuk}, a bounded domain $\Omega\subset\mathbb{C}^n$ has $piecewise\ C^r$-$smooth\ boundary$ if there exists a neighborhood $U$ of $\Omega$ and $C^r$-smooth functions $\rho_1,\cdots,\rho_k:U\rightarrow\mathbb{R}$ such that
\begin{enumerate}
    \item $\Omega=\{z\in U:\rho_1(z)<0,\cdots,\rho_k(z)<0\}$, and $\partial\Omega\subset\cup_{i=1}^kS_i$ with $S_i=\{z\in U:\rho_i(z)=0\}$.

    \item for any possible distinct indices $i_1,\cdots,i_l$, the following
\[
d\rho_{i_1}\wedge d\rho_{i_2}\cdots\wedge d\rho_{i_{l}}\neq0
\]
holds on $S_{i_1}\cap\cdots\cap S_{i_l}$, which means that $d\rho_{i_j}$ are $\mathbb{R}$-linearly independent.
\end{enumerate}
Moreover, a boundary point $p\in\partial\Omega$ is called $generic$ if 
\[
\partial\rho_{i_1}(p)\wedge\cdots\wedge\partial\rho_{i_l}(p)\neq0
\]
whenever $\rho_{i_1}(p)=\cdots=\rho_{i_l}(p)=0$.  This implies that $\partial\rho_{i_j}(p)$ are $\mathbb{C}$-linearly independent. And $\Omega$ is called generic if each boundary point is generic.

Let $\rho:\mathbb{C}^n\rightarrow\mathbb{R}$ be a $C^2$-smooth function such that 
\[
\operatorname{grad}\rho(p)\neq0\ \operatorname{with}\ \rho(p)=0.
\]
We call the hypersurface $H=\{z\in\mathbb{C}^n:\rho(z)=0\}$ is $Levi\ flat$ at $p\in H$ if
\[
\sum_{i,j=1}^n\frac{\partial^2\rho}{\partial z_i\partial\overline{z}_j}(p)w_i\overline{w}_j=0
\]
for any $w=(w_1,\cdots,w_n)$ whenever
\[
\sum_{i=1}^n\frac{\partial\rho}{\partial z_i}(p)w_i=0.
\]
Moreover, $H$ is called Levi flat, if it is Levi flat at every point. For a bounded domain with piecewise $C^2$-smooth boundary  $\Omega=\{z\in U:\rho_1(z)<0,\cdots,\rho_k(z)<0\}$, its boundary $\partial\Omega$ is called $piecewise\ Levi\ flat$ if the hypersurface $\{z\in U:\rho_i(z)=0\}$ is Levi flat for each $i=1,\cdots,k$.

\begin{defn}\label{good corner}
    Let $\Omega\subset\mathbb{C}^n$ be a bounded domain with piecewise smooth boundary. Following \cite{K86}, we say a boundary point $p\in\partial\Omega$ is a $good\ corner$ if there exists a neighborhood $U$ of $p$ and real-valued $C^2$-smooth functions $\rho_1,\cdots,\rho_k$ with $1\leq k\leq n$ such that

\noindent{(C.1)}\ $\rho_1(p)=\cdots=\rho_k(p)=0$;

\noindent{(C.2)}\ $\Omega\cap U=\{z\in U:\rho_i(z)<0,i=1,\cdots,k\}$;

\noindent{(C.3)}\ $\overline{\partial}\rho_1\wedge\cdots\wedge\overline{\partial}\rho_k(z)\neq0$ for all $z\in U$;

\noindent{(C.4)}\ for each $\xi\in T=\{(w_1,\cdots,w_n)\in\mathbb{C}^n:\sum_{j=1}^n\frac{\partial\rho_i}{\partial z_j}(p)w_j=0,i=1,\cdots,k\}$, the following holds
\[
\sum_{j,k=1}^n\frac{\partial^2\rho_i}{\partial z_j\partial\overline{z}_k}(p)\xi_j\overline{\xi}_k\geq0
\]
 for each $i=1,\cdots,k$;

 \noindent{(C.5)}\ for some constant $A\geq0$, the function $\rho=\sum_{i=1}^k\rho_i+A\sum_{i=1}^k\rho^2_i$ is strictly plurisubharmonic on $U$.
\end{defn}
 
\begin{defn}
	Suppose $\Omega=\{z\in U:\rho_1(z)<0,\cdots,\rho_k(z)<0\}$ is a bounded domain with piecewise $C^2$-smooth boundary , and $p$ is a boundary point of $\Omega$. The $index$ of $p$ is defined by
	\[
	r_p:=\#\{i:\rho_i(p)=0\}.
	\]
	A boundary point $p$ is a smooth point if $r_p=1$, and is a singular point otherwise.
\end{defn}

\subsection{Squeezing function}\hfill

In this subsection, we will introduce some properties of squeezing function $S_{\Omega}$ and squeezing function corresponding to polydisk $T_{\Omega}$. More details about $S_{\Omega}$ and $T_{\Omega}$ can be found in \cite{sfp,sque}.
The $squeezing$\ $function$ of a bounded domain $\Omega\subset\mathbb{C}^n$ is defined by 

\begin{align*}
S_{\Omega}(z):=\sup\{r:&\operatorname{there\ exists\ a\ 1-1\ holomorphic\ mapping} f:\Omega\rightarrow\mathbb{B}^n\\
 &\operatorname{with} f(z)=0 \operatorname{and} B_r(0)\subset f(\Omega)\}.
\end{align*}

\begin{rmk}\label{us}
    For a bounded convex domain $\Omega$, \cite[Theorem 1.1]{kz} implies that there exists a constant $c>0$ such that $S_{\Omega}(z)\geq c$ for every $z\in\Omega$.
\end{rmk}

\smallskip

 As an analogy, Gupta and Pant defined the squeezing function corresponding to polydisk in \cite{sfp} as follows.

 \begin{defn}
     Let $\Omega\subset\mathbb{C}^n$ be a bounded domain, the squeezing function corresponding to polydisk is defined by
    \begin{align*}
T_{\Omega}(z):=\sup\{r:&\operatorname{there\ exists\ a\ 1-1\ holomorphic\ mapping} f:\Omega\rightarrow\mathbb{D}^n\\
 &\operatorname{with} f(z)=0 \operatorname{and} D_r(0)\subset f(\Omega)\}.
\end{align*}

 \end{defn}
As a dual notion of $S_{\Omega}$, the boundary behavior of  $T_{\Omega}$ can determine whether a boundary point is infinite type. One can see \cite{hg} for more details.

\subsection{Kernel convergence}$\hfill$

Let $\Omega\subsetneq\mathbb{C}$ be a simply connected domain with a marked point $w$, and let $\Omega_n\subsetneq\mathbb{C}$ be a sequence of simply connected domains with marked points $w_n$. The notion of $Carath\acute{e}odory\ convergence$ is as follows. More details about Carath\'eodory\ convergence can be found in \cite{cc}.
\begin{defn}
    We say that $(\Omega_n,w_n)$ converges to $(\Omega,w)$ in the $Carath\acute{e}odory\ sense$ if the following holds
    \begin{enumerate}
        \item $w_n\rightarrow w$,

        \item any compact set $K\subset\Omega$ is contained in each $\Omega_n$ for large $n$,

        \item for any open connected set $U\ni w$ in $\mathbb{C}$, if $U$ is contained in $\Omega_n$ for infinitely many $n$, then $U\subset\Omega$.
    \end{enumerate}
\end{defn}

Based on the notion of Carath\'eodory convergence, we have the following theorem.

\begin{thm}
    Let $(\Omega,w)$ and $(\Omega_n,w_n)$ be as above. Consider the conformal  Riemann parametrizations $\varphi_n:(\mathbb{D},0)\rightarrow(\Omega_n,w_n)$ with $\varphi_n'(0)>0$ and $\varphi:(\mathbb{D},0)\rightarrow(\Omega,w)$ with $\varphi'(0)>0$. Then the Carath\'eodory convergence of $(\Omega_n,w_n)$ to $(\Omega,w)$ is equivalent to  the uniform convergence of $\varphi_n$ to $\varphi$ on compact sets of $\mathbb{D}$.
\end{thm}
Moreover, we have the following proposition.
\begin{proposition}\label{kc}
    Suppose $f_n=\varphi_n^{-1}$ and $f=\varphi^{-1}$. Then the following are equivalent:
    \begin{enumerate}
        \item $(\Omega_n,w_n)\rightarrow(\Omega,w)$ in the Carath$\acute{e}$odory sense,

        \item for every compact subset $K\subset\Omega$, $K$ is compactly contained in each $\Omega_n$ for large $n$ and $f_n$ converges to $f$ uniformly on $K$.
    \end{enumerate}
\end{proposition}
 
\section{Space of convex domains}\label{sec3}

\begin{defn}
    For compact subsets $A,B$ in $\mathbb{C}^n$, the Hausdorff distance between $A$ and $B$ is defined as
    \[
    d_H(A,B):=\max\left\{\sup_{a\in A}\inf_{b\in B}|a-b|,\sup_{b\in B}\inf_{a\in A}|a-b|\right\}.
    \]
For $R>0$ and a set $A\subset\mathbb{C}^n$, let $A^{(R)}:=A\cap B_R(0)$. The local Hausdorff semi-norm is then defined as
\[
d_H^{(R)}(A,B):=d_H(\overline{A^{(R)}},\overline{B^{(R)}}).
\]
\end{defn}

Let $\mathbb{X}_n$ denote the space of convex domains in $\mathbb{C}^n$ that contains no complex affine line, and $\mathbb{X}_{n,0}$ be the space of pairs $(\Omega,x)$ with $\Omega\in \mathbb{X}_n$ and $x\in\Omega$. Suppose $\operatorname{Aff}(\mathbb{C}^n)$ is the group of affine transformations of $\mathbb{C}^n$, for $A\in\operatorname{Aff}(\mathbb{C}^n)$, the action on $\mathbb{X}_{n,0}$ is given by $A(\Omega,x):=(A\Omega,Ax)$.

\begin{defn}
For a sequence $\left\{\Omega_k\right\}\subset \mathbb{X}_n$, we say that $\Omega_k$ converges to $\Omega_{\infty}$ in $\mathbb{X}_n$ if there exists a constant $R_0>0$ such that 
\[
\lim\limits_{k\rightarrow\infty}d_H^{(R)}(\Omega_k,\Omega_{\infty})=0
\]
for any $R\geq R_0$. Similarly, a sequence  $\left\{\left(\Omega_k,x_k\right)\right\}$ converges to $\left(\Omega_{\infty},x_{\infty}\right)$ in $\mathbb{X}_{n,0}$ if $\Omega_k$ converges to $\Omega_{\infty}$ in $\mathbb{X}_n$ and $x_k\rightarrow x_{\infty}$.
\end{defn}

\smallskip

Since each $\Omega\in\mathbb{X}_n$ is biholomorphic to a bounded domain in $\mathbb{C}^n$, hence $S_{\Omega}$ and $T_{\Omega}$ can be defined on such domain after a biholomorphic transformation. For a convergent sequence in $\mathbb{X}_{n,0}$, both of the squeezing function and the squeezing function corresponding to polydisk have the following upper semi-continuity property.
\begin{proposition}\label{uc}
    Suppose $(\Omega_k,x_k)$ converges to $(\Omega_{\infty},x_{\infty})$ in $\mathbb{X}_{n,0}$, then
    \[
    \limsup_{k\rightarrow\infty}S_{\Omega_k}(x_k)\leq S_{\Omega_{\infty}}(x_{\infty}),\ \limsup_{k\rightarrow\infty}T_{\Omega_k}(x_k)\leq T_{\Omega_{\infty}}(x_{\infty}).
    \]
\end{proposition}

Moreover, we have the following continuity under inclusion.
\begin{proposition}\label{los}
     Suppose $(\Omega_k,x_k)$ converges to $(\Omega_{\infty},x_{\infty})$ in $\mathbb{X}_{n,0}$, and $\Omega_k\subset\Omega_{\infty}$ for large $k$, then
     \[
    \lim_{k\rightarrow\infty}S_{\Omega_k}(x_k)=S_{\Omega_{\infty}}(x_{\infty}),\ \lim_{k\rightarrow\infty}T_{\Omega_k}(x_k)=T_{\Omega_{\infty}}(x_{\infty}).
    \]
\end{proposition}
\begin{proof}
    By Proposition \ref{uc}, it is enough to prove the lower bound. Let $r=T_{\Omega_{\infty}}(x_{\infty})$, note that there exists a holomorphic embedding $f:\Omega_{\infty}\to\mathbb{D}^n$ such that $f(x_{\infty})=0$ and $D_r(0)\subset f(\Omega_{\infty})$. For any $0<r_1<r_2<r$, the compact set $f^{-1}(\overline{D_{r_2}(0)})$ is contained in $\Omega_k$ for large $k$. Since $\Omega_k\subset\Omega_{\infty}$, then $f$ induces a holomorphic embedding from $\Omega_k$ to $\mathbb{D}^n$. Suppose $g_k\in\operatorname{Aut}(\mathbb{D}^n)$ such that $g_k(f(x_k))=0$ and $g_k$ converges to $g(z)=z$ uniformly on compact sets of $\mathbb{D}^n$. Let $h_k=g_k\circ f:\Omega_k\to\mathbb{D}^n$, then $h_k$ is a holomorphic embedding with $h_k(x_k)=0$. Note that $h_k$ converges to $f$ uniformly on compact sets of $\Omega_{\infty}$, hence $D_{r_1}(0)\subset h_k(\Omega_k)$ for large $k$ and $\liminf_{k\to\infty}T_{\Omega_k}(x_k)\geq r_1$. Therefore, we obtain $\lim_{k\to\infty}T_{\Omega_k}(x_k)=T_{\Omega_{\infty}}(x_{\infty})$ after letting $r_1\to r$, and the same discussion is valid for $S_{\Omega}$.
\end{proof}

\smallskip

Suppose $e_1,\cdots,e_n$ is the standard basis of $\mathbb{C}^n$, we define
\[
Z_i:=\left\{e_i+\sum_{j=1}^{n-i}z_je_{i+j}:z_1,\cdots,z_{n-1}\in\mathbb{C}\right\}
\]
for each $1\leq i\leq n$.
Let $\mathbb{K}_n\subset \mathbb{X}_{n}$ be the set of all convex domains $\Omega\in \mathbb{X}_n$ such that $\mathbb{D}e_j\subset\Omega$ and $Z_j\cap\Omega=\emptyset$ for each $j=1,\cdots,n$. We then have the following compactness result.

\begin{thmNoParens}[{\cite[Proposition 4.4, Proposition 4.5]{sc}}]\label{com}
    With the notation above, we have

    (1) $\mathbb{K}_n$ is a compact subset of $\mathbb{X}_{n}$.

    (2) For any $(\Omega,x)\in \mathbb{X}_{n,0}$, there exists some $A\in\operatorname{Aff}(\mathbb{C}^n)$ such that $A(\Omega,x)\in\mathbb{K}_n
    $.
\end{thmNoParens}

\begin{defn} 
For a given convex domain $\Omega\in\mathbb{K}_n$, the vectors $v_1,\cdots,v_n$ are called $\Omega$-$supporting$ if
\[
e_j+\operatorname{Span}_{\mathbb{C}}\{e_{j+1},\cdots,e_n\}\subset\{z\in\mathbb{C}^n:\operatorname{Re}\langle z,v_j\rangle=1\}
\]
and
\[
\Omega\subset\{z\in\mathbb{C}^n:\operatorname{Re}\langle z,v_j\rangle<1\}
\]
for each $j=1,\cdots,n$.
\end{defn}
Then we have the following.

\begin{lemmaNoParens}[{\cite[Lemma 5.1]{sc}}]\label{sup}
    For every $\Omega\in\mathbb{K}_n$, there exist $\Omega$-supporting vectors. If $(v_1,\cdots,v_n)$ are $\Omega$-supporting, then 
    \begin{enumerate}
        \item $v_{j,j}=1$,

        \item $v_{j,k}=0$ if $k>j$,

        \item $|v_{j,k}|\leq1$ if $k<j$.
    \end{enumerate}
\end{lemmaNoParens}
With the notation above, we obtain the following holomorphic embedding.
\begin{lemma}
    The map
    \[
    F(z):=\left(\frac{\langle z,v_1\rangle}{2-\langle z,v_1\rangle},\cdots,\frac{\langle z,v_n\rangle}{2-\langle z,v_n\rangle}\right)
    \]
    is a holomorphic embedding from $\Omega$ into $\mathbb{D}^n$. Moreover, $F(E)=\mathbb{D}^n$ with
    \[
    E:=\left\{z\in\mathbb{C}^n:\operatorname{Re}\langle z,v_1\rangle<1,\cdots,\operatorname{Re}\langle z,v_n\rangle<1\right\}.
    \]
\end{lemma}

\section{Geometry of bounded convex domains}\label{sec4}
 In this section, we study the geometry of bounded convex domains. The main goal of this section is to prove Theorem \ref{main}, Theorem \ref{elf} and Theorem \ref{fv}. The proofs of Theorem \ref{elf} and Theorem \ref{fv} rely on Theorem \ref{main}. At the beginning of this section, we first prove the following local product structure near Levi flat boundary points.

\begin{lemma}\label{lem:local-product}
Let $\Omega\subset\mathbb{C}^2$ be a bounded convex domain, and suppose that $\partial\Omega$ is $C^2$-smooth and Levi flat in a neighborhood of $q\in\partial\Omega$. Then the complex tangent spaces of $\partial\Omega$ have a constant direction in a sufficiently small connected neighborhood of $q$. Moreover, if we choose the coordinates $(z,w)$ centered on $q$ such that
\[
T_q\partial\Omega=\{\operatorname{Im}z=0\},
\]
then there exist constants $r,s>0$ and an open planar set $G_q$ such that
\[
\Omega\cap(B_r(0)\times B_s(0))=(G_q\cap B_r(0))\times B_s(0)
\]
and
\[
\partial\Omega\cap(B_r(0)\times B_s(0))=(\partial G_q\cap B_r(0))\times B_s(0)
\]
\end{lemma}

\begin{proof}
Let $M$ be a sufficiently small connected 
neighborhood of $q$ in $\partial\Omega$, and let $\nu$ denote its outward unit normal. We denote
\[
B_x(X,Y)=\langle D_X\nu,Y\rangle_{\mathbb{R}},
\qquad X,Y\in T_xM,
\]
as the second fundamental form of $M$ with 
\[
D_X\nu=\frac{d}{dt}\Big|_{t=0}\nu\circ\gamma(t),
\]
where $\gamma:(-\epsilon,\epsilon)\to M$ is a smooth curve with $\gamma(0)=x$ and $\gamma'(0)=X$. Since $\Omega$ is convex, then $B_x$ is positive semi-definite.

Now suppose $J$ is the standard complex structure in $\mathbb{C}^2$ and set
\[
E_x=T_xM\cap J(T_xM).
\]
Then $E_x$ is the real two-dimensional complex tangent space.
For every $X\in E_x$, we have that
\[
B_x(X,X)+B_x(JX,JX)=0
\]
because $\partial\Omega$ is Levi flat along $M$.
This implies that $B_x(X,X)=0$. Therefore, we obtain
\[
B_x(X,Y)=0
\qquad\text{for every }X\in E_x,\ Y\in T_xM.
\]
Since $B_x$ is symmetric, then $D_Y\nu$ is orthogonal to $E_x$ for every
$Y\in T_xM$. Moreover,
\[
D_Y\nu\in T_xM,
\qquad
T_xM=E_x\oplus\mathbb{R} J\nu(x).
\]
It follows that
\[
D_Y\nu=\lambda_x(Y)J\nu
      =i\lambda_x(Y)\nu
\]
for a real-valued linear functional $\lambda_x$ on $T_xM$.

 We suppose that $\nu$ is a unit column vector in $\mathbb{C}^2$, and let
\[
P(x)=\nu(x)\nu(x)^*,
\]
where $\nu(x)^*$ denotes conjugate transpose of $\nu(x)$. Then for each $Y\in T_xM$,
\[
D_YP
=
i\lambda_x(Y)\nu\nu^*
-i\lambda_x(Y)\nu\nu^*
=0.
\]
This implies that $P$ is constant in $M$ and hence
the complex tangent direction is constant.

As the discussion above, the complex tangent direction in
$\{0\}\times\mathbb{C}$ in $M$. By the implicit function theorem, after shrinking the neighborhood if necessary, the boundary $M$ can be written as
\[
\operatorname{Im}z
=
f(\operatorname{Re}z,\operatorname{Re}w,\operatorname{Im}w)
\]
near $q$. Since $T_p^{\mathbb{C}}M=\{0\}\times\mathbb{C}$ for each $p\in M$, then we obtain that
\[
\frac{\partial f}{\partial(\operatorname{Re}w)}
=
\frac{\partial f}{\partial(\operatorname{Im}w)}
=0.
\]
And hence the definition function of $\partial\Omega$ at $q$ can be written as
\[
\operatorname{Im}z=f(\operatorname{Re}z),
\]
that is there exists a neighborhood $U$ of $q$ such that
\[
\Omega\cap U=\{(z,w)\in U:\operatorname{Im}z>f(\operatorname{Re}z)\}.
\]
Now let $G_q=\{z\in\mathbb{C}:\operatorname{Im}z>f(\operatorname{Re}z)\}$, then $G_q$ is an open set of $\mathbb{C}$ and $0\in\partial G_q$. If we choose $r,s>0$ such that $B_r(0)\times B_s(0)\subset U$, then we obtain
\[
\Omega\cap(B_r(0)\times B_s(0))=(G_q\cap B_r(0))\times B_s(0)
\]
and
\[
\partial\Omega\cap(B_r(0)\times B_s(0))=(\partial G_q\cap B_r(0))\times B_s(0).
\]
\end{proof}

Now we prove the singular part of Theorem \ref{main}.
\begin{thm}\label{ls}
    Let $\Omega\subset\mathbb{C}^2$ be a bounded convex domain with piecewise Levi flat boundary. Suppose $p\in\partial\Omega$ is a generic singular point, then  
    \[
    \lim\limits_{z\rightarrow p}T_{\Omega}(z)=1.
    \]
\end{thm}
\begin{proof}
Since $p\in\partial\Omega$ is a generic singular boundary point, after changing the coordinates, we may assume that $p=0$ and there exists a neighborhood $U$ of $p$ such that
\[
\Omega\cap U=\{z\in\mathbb{C}^2:\rho_{1}(z)<0,\rho_{2}(z)<0\}\cap U
\]
with $d\rho_1(0)=-d(\operatorname{Im}z_1),d\rho_2(0)=-d(\operatorname{Im}z_2)$.
\smallskip

\noindent\textbf{Claim}: If
\[
D_i=\pi_i(\Omega),\qquad i=1,2,
\]
are the coordinate projections of $\Omega$, then
$D_i$ are bounded convex planar domains and $0$ is a $C^2$-smooth
boundary point of each $D_i$. Furthermore, for some neighborhood $W$ of $0$, we have the following
\begin{equation}\label{eq:local-product}
\Omega\cap W=(D_1\times D_2)\cap W.
\end{equation}
\smallskip

\noindent$Proof\ of\ Claim:$ Let $S_i=\{\rho_i=0\}\cap U$ and $M_i:=\{z\in\mathbb{C}^2:\rho_{i}(z)=0,\rho_{3-i}(z)<0\}\cap U$, we may assume that $M_i$ is connected after shrinking $U$. Then
Lemma \ref{lem:local-product} implies that the complex tangent
direction is  constant on $ M_i$. Taking limits at 0 yields that for each $q_i\in M_i$, then
\[
T^{\mathbb{C}}_{q_i}S_i=\operatorname{ker}dz_i.
\]
By the continuity of the tangent spaces up to $0$, we obtain the following

\[
T^{\mathbb{C}}_{0}S_1=\ker dz_1,\qquad
T^{\mathbb{C}}_{0}S_2=\ker dz_2.
\]

Now let $z_i=x_i+iy_i$. By the implicit function theorem, the hypersurface $S_1$ can be written as
\[
y_1=f_1(x_1,x_2,y_2)
\]
in $F=\{(x_1,x_2,y_2)\in\mathbb{R}^3:|x_1|<a_0,|x_2|<b_0,|y_2|<c_0\}$ for small $a_0,b_0,c_0>0$.
Moreover, we also have
\begin{equation}\label{eq:vanishing-derivatives}
\frac{\partial f_1}{\partial x_2}
=
\frac{\partial f_1}{\partial y_2}
=0
\end{equation}
on $M_1$. Now we restrict $\rho_2$ to $S_1$ and define
\[
R(x_1,x_2,y_2)=\rho_2(x_1+if_1(x_1,x_2,y_2),x_2+iy_2),
\]
then $\frac{\partial R}{\partial y_2}(0)=-1$. After shrinking $F$, there exists $C^2$ function $g$ such that
\[
R=0\ \operatorname{if\ and\ only\ if}\ y_2=g(x_1,x_2)
\]
and
\[
R<0\ \operatorname{if\ and\ only\ if}\ y_2>g(x_1,x_2).
\]
For fixed $c\in(0,c_0)$, we can chosen $a,b$ small enough such that $g(x_1,x_2)<\frac{c}{4}$. Let $\tau=\frac{c}{2}$ and 
\[
F_{x_1}:=\{(x_2,y_2):|x_2|<b,g(x_1,x_2)<y_2<c\}
\]
for fixed $x_1$. For each $(x_2,y_2)\in F_{x_1}$, we can move it to $(x_2,\tau)$ along a vertical segment $l_1$ and then move $(x_2,\tau)$ to $(0,\tau)$ along a horizontal segment $l_2$. Note that both $l_1,l_2$ are contained in $F_{x_1}$. Integrating (\ref{eq:vanishing-derivatives}) along $l_1$ and $l_2$, we obtain
\[
f_1(x_1,x_2,y_2)=f_1(x_1,0,\tau)=:h_1(x_1).
\]
Then $h_1$ is $C^2$ smooth and $M_1\subset\{y_1=h_1(x_1)\}$. Moreover, $M_2\subset\{y_2=h_2(x_2)\}$ for some $C^2$ smooth function $h_2$. Taking limits to the common edge yields
\begin{equation}\label{eq:boundary-cylinders}
h_i(0)=h_i'(0)=0,\quad i=1,2,
\end{equation}
and every boundary point sufficiently close to $0$ belongs
to one of these two cylinders.

For each small real $s$, the above construction implies  $q_s=(s+ih_1(s),i\tau)\in M_1$.
Since its tangent hyperplane supports the entire convex domain, hence
\begin{equation}\label{eq:global-support}
\Omega\subset
\left\{z:
\operatorname{Im} z_1>
h_1(s)+h_1'(s)(\operatorname{Re} z_1-s)
\right\}.
\end{equation}
After applying the same argument, we obtain
\begin{align}\label{sup}
\Omega\subset
\left\{z:
\operatorname{Im} z_i>
h_i(s)+h_i'(s)(\operatorname{Re} z_i-s)
\right\},\ i=1,2.
\end{align}
In particular, the functions $h_i$ are convex in a neighborhood of $0$.

For small $a,b>0$, we let
\[
Q_i=\{x_i+iy_i:|x_i|<a,|y_i|<b\}
\] 
and
\[
G_i=\{x_i+iy_i\in Q_i:y_i>h_i(x_i)\}.
\]
After shrinking $a,b$, we may assume the preceding
descriptions hold on $W=Q_1\times Q_2$. Let $P=G_1\times G_2$, then $P$ is connected because $h_i$ is convex. After applying \eqref{eq:global-support} with
$s=\operatorname{Re} z_i$, then
\[
\Omega\cap W
\subset
P.
\]
It contains points
of $\Omega$ arbitrarily close to $0$.
By \eqref{eq:boundary-cylinders}, $P$ contains no point of
$\partial\Omega$. Thus $\Omega\cap P$ is both open and closed
in $P$, and is nonempty. Hence $P\subset\Omega$ and
$\Omega\cap W=P$.

Finally, set $D_i=\pi_i(\Omega)$. Then each $D_i$ is a bounded convex domain in $\mathbb{C}$. Since $\Omega\cap W= P$, then for each $\zeta\in G_i$ we know that $\zeta\in D_i$. Moreover, let $\zeta=x+iy\in D_i\cap Q_i$, then there exists $\xi\in\Omega$ such that $\pi_i(\xi)=\zeta$. Let $s=x$ in (\ref{sup}), then 
\[
y>h_i(x)
\]
and this implies that $\zeta\in G_i$. The above discussion means that $D_i\cap Q_i=G_i$ and hence
\[
(D_1\times D_2)\cap W=(D_1\cap Q_1)\times(D_2\cap Q_2)=G_1\times G_2=\Omega\cap W.
\]
Moreover, the definition function of $D_i$ can be written as
\[
r_i(x+iy)=h_i(x)-y
\]
near $0$ and $\nabla r_i=(h'_i(x),-1)\neq0$, this implies that $0$ is a smooth boundary point of $D_i$.$\hfill\square$

\medskip

Let $p_k=(p_{1,k},p_{2,k})\in\Omega$ tend to $0$. Define
\begin{equation}\label{eq:planar-distances}
d_{i,k}=\operatorname{dist}(p_{i,k},\partial D_i),
\qquad
q_{i,k}\in\partial D_i,
\qquad
|q_{i,k}-p_{i,k}|=d_{i,k}.
\end{equation}
Since $0\in\partial D_i$,
\[
d_{i,k}\leq |p_{i,k}|\rightarrow0,
\qquad
|q_{i,k}|\leq2|p_{i,k}|\rightarrow0.
\]
Hence, for large $k$, the points $q_{i,k}$ lie on the smooth
boundary arcs near $0$, and
$p_{i,k}-q_{i,k}$ is in the inward normal vector to $\partial D_i$ at $q_{i,k}$.

After passing to a subsequence and exchanging the coordinates, we may assume that
\[
d_{1,k}\leq d_{2,k}.
\]
Let
\[
T_k(z)=z-p_k.
\]
By \eqref{eq:local-product},
\[
Q_{1,k}=(q_{1,k},p_{2,k}),
\qquad
Q_{2,k}=(p_{1,k},q_{2,k})
\]
are actual smooth boundary points of $\Omega$ for all sufficiently large $k$. Suppose
\[
v_{1,k}=(q_{1,k}-p_{1,k},0),
\qquad
\lambda_{1,k}=d_{1,k}.
\]
\[
v_{2,k}=(0,q_{2,k}-p_{2,k}),
\qquad
\lambda_{2,k}=d_{2,k}.
\]
By the claim and $d_{1,k}\leq d_{2,k}$, we know that $\lambda_{1,k}=\delta_{\Omega}(p_k)$.

Choose $u_{i,k}\in\mathbb{C}$, $|u_{i,k}|=1$, such that
\[
u_{i,k}(q_{i,k}-p_{i,k})=d_{i,k}.
\]
Set
\[
U_k=\operatorname{diag}(u_{1,k},u_{2,k}),
\qquad
B_k=\operatorname{diag}(d_{1,k}^{-1},d_{2,k}^{-1}),
\]
and the affine map $A_k=B_kU_kT_k$ as
\begin{equation}\label{eq:normalization}
A_k(z_1,z_2)=
\left(
\frac{u_{1,k}(z_1-p_{1,k})}{d_{1,k}},
\frac{u_{2,k}(z_2-p_{2,k})}{d_{2,k}}
\right).
\end{equation}
Then $A_k(p_k)$=0 and $A_k(Q_{i,k})=e_i$ for $i=1,2$. Set $\Omega_k=A_k\Omega$, then $\Omega_k\in\mathbb{K}_2$. By Theorem \ref{com}, after passing to a subsequence, $\Omega_k$ converges to some $\Omega_{\infty}$ in $\mathbb{X}_2$. Since $D_i$ is convex and the tangent line at $q_{i,k}$ supports $D_i$, then
\begin{equation}\label{eq:outer-inclusion}
\Omega_k\subset
E:=\{w\in\mathbb{C}^2:\operatorname{Re} w_1<1,\ \operatorname{Re} w_2<1\}.
\end{equation}

Now we describe the shape of $\Omega_{\infty}$. For each $k$, let
\[
(v_{1},v_{2})=U_kT_k(z_1,z_2)=(u_{1,k}(z_1-p_{1,k}),u_{2,k}(z_2-p_{2,k}))
\]
and $D_{i,k}:=u_{i,k}(D_i-p_{i,k})$. Then we have $v_i-d_{i,k}=u_{i,k}(z_i-q_{i,k})$ because $u_{i,k}(q_{i,k}-p_{i,k})=d_{i,k}$. Note that $q_{i,k}-p_{i,k}$ is in the outer normal direction of $\partial D_i$ at $q_{i,k}$. Therefore, the tangent line to $\partial D_{i,k}$ at $d_{i,k}$ is $\operatorname{Re}v_i=d_{i,k}$ and
\[
D_{i,k}\subset\{v\in\mathbb{C}:\operatorname{Re}v<d_{i,k}\}.
\]
Since $D_i$ is determined by $\{r_i<0\}$ in a neighborhood of $0$, let $f_{i,k}=r_{i}(q_{i,k}+u_{i,k}^{-1}(s+it))$ be the defining function near $q_{i,k}$ and $v_i=d_{i,k}+s+it$, then we obtain that
\[
f_{i,k}(0,0)=0,\ |\partial_sf_{i,k}(0,0)|=|\nabla r_i(q_{i,k})|>0,\ \partial_t f_{i,k}(0,0)=0
\]
from the choice of the normal direction.
Since $q_{i,k}\to 0$, then the second derivatives of $f_{i,k}$ is bounded in a neighborhood of $(0,0)$ and $\partial_s f_{i,k}$ is bounded away from 0. Hence the implicit function theorem implies that there exists $C^2$ function $g_{i,k}$ such that the boundary and interior of $D_{i,k}$ near $d_{i,k}$ can be described by
\[
\operatorname{Re} v_{i}=d_{i,k}-g_{i,k}(\operatorname{Im} v_{i}),
\qquad
\operatorname{Re} v_{i}<d_{i,k}-g_{i,k}(\operatorname{Im} v_{i}),
\]
respectively in a neighborhood of $v_i=d_{i,k}$. From the above construction, we also obtain 
\begin{equation}\label{eq:quadratic-estimate}
g_{i,k}(0)=g_{i,k}'(0)=0,
\qquad
0\leq g_{i,k}''(t)\leq C
\quad (|t|<r_0)
\end{equation}
with some constants $r_0,C>0$ that are independent of large $k$. Note that $g_{i,k}(t)=t^2\int_{0}^1(1-\theta) g''_{i,k}(\theta t)d\theta$ means that $g_{i,k}(t)\leq Ct^2$ for $|t|<r_0$.

If $(w_1,w_2)=B_k(v_{1},v_{2})$, then
\[
v_{i}=d_{i,k}w_i.
\]
For fixed $R>0$, since $A_k^{-1}(B_R(0))$ converges to $0$, hence the domain $\Omega_k\cap B_R(0)$ can be described as 
\begin{equation}\label{eq:scaled-inequalities}
\operatorname{Re} w_i<
1-\frac{g_{i,k}(d_{i,k}\operatorname{Im} w_i)}{d_{i,k}},
\qquad i=1,2.
\end{equation}
for large $k$.
For every fixed $R>0$ and $w\in B_R(0)$, \eqref{eq:quadratic-estimate} gives
\begin{equation}\label{eq:vanishing-error}
0\leq
\frac{g_{i,k}(d_{i,k}\operatorname{Im} w_i)}{d_{i,k}}
\leq C d_{i,k}R^2\rightarrow0.
\end{equation}
Taking limit for $k$, after passing to a subsequence, we know that $\Omega_k$ converges to $E$ in the local Hausdorff distance. 

Now we consider the Cayley map
\[
F:E\rightarrow\mathbb{D}^2,
\qquad
F(w)=
\left(\frac{w_1}{2-w_1},\frac{w_2}{2-w_2}\right),
\]
which is a biholomorphism bewteen $E$ and $\mathbb{D}^2$ satisfying $F(0)=0$. Since $\Omega_k\subset E$, then $F$ induces a holomorphic embedding
\[
F:\Omega_k\rightarrow\mathbb{D}^2
\]
with
$F(0)=0$. Because $(\Omega_k,A_kp_k)$ converges to $(E,0)$ in $\mathbb{X}_{2,0}$, then Proposition \ref{los} implies that
\[
\lim_{k\to\infty}T_{\Omega}(p_k)=1.
\]

Due to the arbitrariness of $p_k$, we obtain that $\lim\limits_{z\rightarrow p}T_{\Omega}(z)=1$. To see this, we assume that there exists $s_j\in\Omega$ converges to $p$ such that
\[
\lim_{j\rightarrow\infty}T_{\Omega}(s_j)=c<1.
\]
By the previous argument, we know that there exists a subsequence $s_{j_k}$ such that 
\[
\lim_{k\rightarrow\infty}T_{\Omega}(s_{j_k})=1,
\]
this contradicts our assumption and completes the proof.
\end{proof}

\smallskip

From \cite{fol}, we know that a Levi flat hypersurface can be foliated by codimension one holomorphic foliation.

\begin{proposition}\label{foli}
    Let $\Omega:=\{z\in\mathbb{C}^n:r(z)<0\}(n\geq2)$ be a bounded domain with $C^2$-smooth boundary. For a boundary point $p\in\partial\Omega$, if there exists a neighborhood $U$ of $p$ such that any point in $U\cap\partial\Omega$ is Levi flat, then there exists a non-constant holomorphic map $f:\mathbb{D}\rightarrow\partial\Omega$ such that $f(0)=p$.
\end{proposition}

For a convex domain $\Omega$, we recall that if there exists a holomorphic map $f:\mathbb{D}\rightarrow \partial\Omega$ containing $p\in\partial\Omega$, then there exists an affine disk contained in $\partial\Omega$ through $p$. Now we prove the smooth part of Theorem \ref{main}, which can be deduced from the following theorem.

\begin{thm}\label{sp}
   Suppose $\Omega\subset\mathbb{C}^2$ is a bounded convex domain with piecewise Levi flat boundary, and $p\in\partial\Omega$ is a smooth boundary point. Then for each sequence $\{p_k\}\subset\Omega$ converging to $p$, there exists $A_k\in\operatorname{Aff}(\mathbb{C}^2)$ such that, after passing to a subsequence, $A_k(
\Omega,p_k)$ converges to $(\Omega_{\infty},0)$ in $\mathbb{X}_{2,0}$ and $\Omega_{\infty}$ is biholomorphic to the bidisk $\mathbb{D}^2$ .
\end{thm}
\begin{proof}

 Let $p\in\partial\Omega$ satisfy the hypothesis  and let $p_j\in\Omega$ converge to $p$. Without loss of generality we assume that $p=(0,0)$ and
\begin{enumerate}
    \item  $\Omega\subset\left\{(z_1,z_2)\in\mathbb{C}^2:\operatorname{Im}z_1>0\right\}$,

\item $\{0\}\times\mathbb{D}\subset\partial\Omega$,

\item  $(i,0)\in\Omega$.
\end{enumerate}
For each $j$, we can find $q_j\in\partial\Omega$ such that
\[
|q_j-p_j|=\min_{x\in\partial\Omega}|p_j-x|.
\]
Note that since $\partial\Omega$ is $C^2$-smooth around $p$, then such $q_j$ is uniquely determined for large $j$, and $p_j$ lies along the interior normal direction at $q_j$.

Let  
\[
S=\{w\in\mathbb{C}:(0,w)\in\overline{\Omega}\},\ C_1=\{z\in\mathbb{C}:(z,0)\in\Omega\}
\]
and $C_2=\operatorname{int}S$. Then $C_1$ is $C^2$-smooth at $0\in\partial C_1$. 
\medskip

\noindent\textbf{Claim 1}: For each $\eta\in C_2$, there exists some constant $r_\eta>0$ and an open neighborhood
$V_\eta\subset\subset C_2$ of $\eta$ such that
\begin{equation}\label{eq:pointwise-product}
\Omega\cap(B_{r_{\eta}}(0)\times V_\eta)
=
(C_1\cap B_{r_{\eta}}(0))\times V_\eta.
\end{equation}
Moreover, for every compact set $K\subset\subset C_2$ there exists a constant
$r_K>0$ and an open set $V_K$ satisfying
\[
K\subset\subset V_K\subset\subset C_2
\]
and
\begin{align}
\Omega\cap(B_{r_K}(0)\times V_K)
&=(C_1\cap B_{r_K}(0))\times V_K,
\label{eq:compact-product}\\
\partial\Omega\cap(B_{r_K}(0)\times V_K)
&=(\partial C_1\cap B_{r_K}(0))\times V_K.
\label{eq:compact-boundary}
\end{align}
\smallskip

\noindent$Proof\ of\ Claim\ 1$: Fix $\eta\in C_2$ and put $q=(0,\eta)$. We now prove that $q$ is a smooth point in $\partial\Omega$. For each $j\in I(q)$ with $I(q)=\{j:\rho_j(q)=0\}$,
the real hyperplane
\[
H_j=\{z:d\rho_j(q)(z-q)=0\}
\]
supports $\Omega$ at $q$. Indeed, if $w\in\Omega$, then  $q+t(w-q)\in\Omega$ for $0<t\leq1$ because $\Omega$ is convex. Hence
\begin{align}\label{spi}
    d\rho_j(q)(w-q)
=
\lim_{t\to0}
\frac{\rho_j(q+t(w-q))-\rho_j(q)}{t}
\leq0.
\end{align}
The covector $d\rho_j(q)$ is nonzero by the transversality
assumption.

Since $\eta$ is an interior point of $C_2$, then for every $v\in\mathbb{C}$
and all sufficiently small $t>0$ we obtain that
\[
q+(0,tv),\ q-(0,tv)\in\overline{\Omega}.
\]
After applying (\ref{spi}) to these two points, then
\[
d\rho_j(q)((0,v))=0
\qquad(v\in\mathbb{C}).
\]
Thus $H_j$ contains the entire affine complex line
\[
L_0=\{0\}\times\mathbb{C}.
\]
In particular, it passes through $p$ and is also a supporting
hyperplane of $\Omega$ at $p$.

Because $p$ is a smooth point, the supporting hyperplane at $p$ is unique. From our assumption, the supporting hyperplane at $p$ is
\[
H_0=\{\operatorname{Im}z_1=0\}.
\]
Therefore $H_j=H_0$ for every $j\in I(q)$ and hence
\[
d\rho_j(q)=-\lambda_j\,d(\operatorname{Im}z_1),
\qquad \lambda_j>0.
\]
If there are two or more elements in $I(q)$, then their real
differentials would be linearly dependent. Hence $I(q)$ consists
of exactly one element. This follows that $\partial\Omega$ is a single $C^2$-smooth hypersurface near $q$ and the global piecewise Levi flat assumption makes this hypersurface Levi flat. Moreover, the complex tangent direction of $\partial\Omega$ at $q$ is $\{0\}\times\mathbb{C}$.

Now we apply Lemma~\ref{lem:local-product} at each
$q=(0,\eta)$, $\eta\in C_2$. Since the complex tangent direction
at $q$ is the fixed $w$-direction, then there exists a constant $r_\eta>0$, an open
neighborhood $V_\eta\subset\subset C_2$ of $\eta$, and an open planar
set $G_\eta$ such that
\begin{equation}\label{eq:temporary-factor}
\Omega\cap(B_{r_{\eta}}(0)\times V_\eta)
=
(G_\eta\cap B_{r_{\eta}}(0))\times V_\eta.
\end{equation}
We now show that $G_{\eta}$ coincides with $C_1$. Fix $\eta\in C_2$. Since $C_2$ is a bounded convex domain and $0\in C_2$, then
\[
[0,\eta]\subset\subset C_2.
\]
For each $z\in[0,\eta]$, we choose the neighborhood $V_{z}$ in \eqref{eq:temporary-factor} and the union of $V_z$ can cover the segment $[0,\eta]$. Since $[0,\eta]$ is compact and connected, then we can choose a finite chain
$V_0,\ldots,V_N$ such that
\[
0\in V_0,\qquad
\eta\in V_N,\qquad
V_k\cap V_{k+1}\neq\varnothing
\quad(0\leq k<N).
\]
We write $r_k$ and $G_k$ for the corresponding radii and planar
sets appeared in \eqref{eq:temporary-factor}, and set
\[
r=\min_{0\leq k\leq N}r_k>0.
\]
For each $k<N$, choose $\xi_k\in V_k\cap V_{k+1}$.
For every $z\in B_r(0)$, after applying \eqref{eq:temporary-factor} to both $G_k$ and $G_{k+1}$, then
\[
z\in G_k
\quad\Longleftrightarrow\quad
(z,\xi_k)\in\Omega
\quad\Longleftrightarrow\quad
z\in G_{k+1}.
\]
This means that
\[
G_k\cap B_r(0)=G_{k+1}\cap B_r(0),
\]
and 
\[
G_N\cap B_r(0)=G_0\cap B_r(0)
\]
by induction.
Because $0\in V_0$, then the definition of $C_1$ shows that
\[
G_0\cap B_r(0)
=
\{z\in B_r(0):(z,0)\in\Omega\}
=
C_1\cap B_r(0).
\]
After applying \eqref{eq:temporary-factor} to $B_r(0)\times V_N$, then
\[
\Omega\cap(B_r(0)\times V_N)=(C_1\cap B_r(0))\times V_N,
\]
which proves \eqref{eq:pointwise-product}.

Now let $K\subset\subset C_2$ be a compact subset. We can choose finitely many
neighborhoods $V_{\eta_1},\ldots,V_{\eta_l}$ for which
\eqref{eq:pointwise-product} holds with $\overline{V}_{\eta_j}\subset C_2$ such that
\[
K\subset\bigcup_{j=1}^l V_{\eta_j}.
\]
Define
\[
r_K=\min_{1\leq j\leq l}r_{\eta_j},
\qquad
V_K=\bigcup_{j=1}^{l}V_{\eta_j}.
\]
Then $K\subset\subset V_K\subset\subset C_2$.  Note that for each $j=1,\cdots,l$, we have
\[
\Omega\cap(B_{r_K}(0)\times V_{\eta_j})=(C_1\cap B_{r_K}(0))\times V_{\eta_j}.
\]
After taking union in both sides, we obtain \eqref{eq:compact-product}. By taking relative boundaries in the open set
$B_{r_K}(0)\times V_K$, then \eqref{eq:compact-boundary} holds for such $r_{K}$ and $V_K$.
$\hfill\square$

\medskip

Let
\[
p_j=(z_j,w_j)\in\Omega,
\qquad p_j\rightarrow p.
\]
For large $j$, let $q_j$ be a nearest boundary point to $p_j$,
and write
\[
\delta_j=|p_j-q_j|.
\]
From Claim 1, we obtain that
\[
q_j=(b_j,w_j),
\qquad b_j\in\partial C_1,
\qquad b_j\rightarrow0.
\]
Indeed, $p_j-q_j$ lies in the inner normal direction of $\partial\Omega$ at $q_j$, and the normal direction has no $w$-component.

For each $j$, we choose unit complex numbers $\alpha_j$ such that
\[
\alpha_j(z_j-b_j)=i\delta_j,
\qquad
\alpha_j\rightarrow1,
\]
If we define $\psi_j$ as
\begin{equation}\label{eq:unitary-normalization}
\psi_j(z,w)
=
\bigl(\alpha_j(z-b_j),w-w_j\bigr),
\end{equation}
then we obtain
\[
\psi_j(\Omega)\subset\left\{(z_1,z_2)\in\mathbb{C}^2:\operatorname{Im}z_1>0\right\}
\]
and $\psi_j(p_j)=(\delta_ji,0)$ with some constant $\delta_j>0$ for large $j$. Because $|p_j-q_j|\rightarrow0$, we have $\delta_j\rightarrow0$.

For each large $j$, we pick $\xi_j\in\left(\{\delta_ji\}\times\mathbb{C}\right)\cap\partial\psi_j(\Omega)$ such that 
\[
|\xi_j-\psi_j(p_j)|=\min\left\{|\xi-\psi_j(p_j)|:\xi\in\left(\{\delta_ji\}\times\mathbb{C}\right)\cap\partial\psi_j(\Omega)\right\}.
\]
The boundedness of $\Omega$ implies that $\limsup\limits_{j\rightarrow\infty}|\xi_j-\psi_j(p_j)|<\infty$. Suppose $\xi_j=(\delta_ji,a_j)$, after passing to a subsequence, we assume $a_j\rightarrow a$. Then we have 
\[
\lim\limits_{j\rightarrow\infty}|\xi_j-\psi_j(p_j)|=\lim\limits_{j\rightarrow\infty}|a_j|=|a|
\]
and $(0,a)\in\partial\Omega$. Note that after passing to a subsequence, $\psi_j$ converges to the identity map $I_2:\mathbb{C}^2\rightarrow\mathbb{C}^2$. Since $\{0\}\times\mathbb{D}\subset\partial\Omega$ and $\Omega$ is convex, we may assume that $|a|\geq 1$. 

Let 
\begin{align*}
    A_j:=\left(\begin{matrix}
         \frac{1}{\delta_j}&0  \\
         0&a_j^{-1} 
    \end{matrix}\right).    
\end{align*}
and the affine map $T\in\operatorname{Aff}(\mathbb{C}^2)$
\[
T(z):=(i(z_1-i),z_2),
\]
then $T\circ A_j(\psi_j(\Omega),\psi_j(p_j))\in\mathbb{K}_2$ after combining Claim 1, which means that $A_j\psi_j(\Omega)$ converges to some $\Omega_1$ in $\mathbb{X}_2.$  We define $D_2:=a^{-1}C_2$ and $H:=\{z\in\mathbb{C}:\operatorname{Im}z>0\}$.
    
\smallskip

\noindent\textbf{Claim 2}:  $\{0\}\times D_2\subset\partial\Omega_1$ and $\Omega_1\subset H\times D_2$.

\smallskip

\noindent$Proof\ of\ Claim\ 2$: If $(x,y)\in\Omega_1$, then there exists $(x_j,y_j)\in\Omega$ such that $A_j\psi_j(x_j,y_j)\rightarrow(x,y)$. Let $(c_j,d_j)=\psi_j(x_j,y_j)$, then we obtain $\frac{c_j}{\delta_j}\rightarrow x$ and $a_j^{-1}d_j\rightarrow y$, it means that $c_j\rightarrow0$ and $d_j\rightarrow ay$. Since $\psi_j\rightarrow I_2$, then $x_j\rightarrow 0$ and $y_j\rightarrow ay$. Hence $y\in a^{-1}C_2$ and $\Omega_1\subset\mathbb{C}\times D_2$. As  $\psi_j(\Omega)\subset H\times\mathbb{C}$ for each $j$ and $(i,0)\in\Omega_1$, we also have $\Omega_1\subset H\times\mathbb{C}$ and 
\[
\Omega_1\subset (H\times\mathbb{C})\cap(\mathbb{C}\times D_2)=H\times D_2.
\]

Fix $y\in D_2$ and put
\[
\eta=ay\in C_2,
\qquad
\eta_j=w_j+a_jy.
\]
Then $\eta_j\to\eta$. Let
$K\subset\subset C_2$ be a compact set of $\eta$ such that
$\eta_j\in K$ and $b_j\in\partial C_1\cap B_{r_K}(0)$ for large $j$.
By \eqref{eq:compact-boundary},
\[
Q_j(y):=(b_j,\eta_j)\in\partial\Omega.
\]
Note that
\[
A_j\psi_j\bigl(Q_j(y)\bigr)
=
A_j(0,a_jy)
=
(0,y).
\]
Thus $(0,y)\in\partial\Omega_j$ for all sufficiently large $j$.
Therefore,
$(0,y)\in\overline{\Omega}_1$.
Since
$\Omega_1\subset H\times D_2$, then $(0,y)\in\partial\Omega_1$.
And this completes the proof because $y$ is arbitrary.
$\hfill\square$
\medskip

 Now let $C_{1,j}\subset\mathbb{C}$ be the open convex set such that 
 \[
 C_{1,j}\times\{0\}=(\mathbb{C}\times\{0\})\cap\psi_j(\Omega),
 \]
then $C_{1,j}$ converges to $C_1$ in $\mathbb{X}_1$. Indeed, let $z_j\in C_{1,j}$ such that $z_j\to z$, then $(z_j,0)\in \psi_j(\Omega)$ and hence $\psi_j^{-1}(z_j,0)\in\Omega$. Since $\psi_j$ converges to $I_2$, then $(z,0)\in\Omega$. On the other hand, for each $z\in C_1$, there exists $\epsilon>0$ such that $B_{\epsilon}((z,0))\subset\Omega$. Now $\psi_j^{-1}((z,0))\to(z,0)$ implies that $\psi^{-1}_j((z,0))\in B_{\epsilon}((z,0))\subset\Omega$ for large $j$. This means that $(z,0)\in\psi_j(\Omega)$ and $z\in C_{1,j}$ for large $j$. Moreover, since  $ B_r(i)\subset\delta_{j}^{-1}C_{1,j}\subset H$ holds for large $j$ with some $r>0$, then, after passing to a subsequence, $\delta_j^{-1}C_{1,j}$ converges to some open convex set $D_1\in \mathbb{X}_1$.

\medskip

\noindent\textbf{Claim 3}: $H\times D_2=\Omega_1$.
\smallskip

\noindent$Proof\ of\ Claim\ 3$: By construction, we have
\[
(\delta_j^{-1}C_{1,j})\times\{0\}\subset A_j\psi_j(\Omega),    
\]
which implies that $D_1\times\{0\}\subset\overline{\Omega}_1$ from the definition of local Hausdorff convergence.

Now we prove $H=D_1$. Let 
\[
E:=\bigcup_{j=1}^{\infty}\delta_j^{-1}C_1,
\]
then $E\subset D_1$. To see this, suppose $x\in E$, then $x\in\delta_j^{-1}C_1$ for some $j\geq1$. This means that $\delta_jx\in C_1$. Since $C_{1,j}$ converges to $C_1$ in local Hausdorff distance, then $\delta_jx\in C_{1,k}$ for large $k$. We may assume that $k\geq j$ and $\delta_k^{-1}\geq\delta_j^{-1}$, then $x\in\delta_k^{-1}C_{1,k}$ since $C_{1,k}$ is convex and $0\in\partial C_{1,k}$. Hence, we obtain that $x\in D_1$. Moreover, since $\partial\Omega$ is $C^2$-smooth around $0$, then $C_1$ is $C^2$-smooth around $0$ and this deduces that $E=H$ by \cite[Lemma A.8]{sc}. Therefore, we have $H=E\subset D_1\subset H$.

For any $(x,y)\in D_1\times D_2$, since $D_1$ is a cone and $(\delta_j^{-1}x,0)\in\overline{\Omega}_1$ for each $j$, then
\[
(x,y)=\lim_{j\rightarrow\infty}\frac{1}{\delta_j^{-1}}(\delta_j^{-1}x,0)+\frac{\delta_j^{-1}-1}{\delta_j^{-1}}(0,y)\in\overline{\Omega}_1
\]
provided $(0,y)\in\overline{\Omega}_1$. Thus $D_1\times D_2\subset\overline{\Omega}_1$. Because $\Omega_1$ is convex and open, then we obtain $D_1\times D_2\subset\Omega_1$. Hence 
\[
D_1\times D_2\subset\Omega_1\subset H\times D_2
\]
and
\[
\Omega_1=H\times D_2.
\]
$\hfill\square$

Since $D_2$ is a bounded convex domain in $\mathbb{C}$, then $D_2$ is biholomorphic to $\mathbb{D}$ by Riemann mapping theorem, and hence $H\times D_2$ is biholomorphic to $\mathbb{D}^2$. With the discussion above, this completes the proof.
\end{proof}
With Theorem \ref{sp} in hand, we now prove the smooth part of Theorem~\ref{main}.
\begin{thm}\label{sp1}
    Let $\Omega\subset\mathbb{C}^2$ be a bounded convex domain with piecewise Levi flat boundary. If $p\in\partial\Omega$ is a smooth point, then 
    \[
    \lim_{z\rightarrow p}T_{\Omega}(z)=1.
    \]
\end{thm}
\begin{proof}
Let $\Omega,A_j,\psi_j,H,D_2$ be as in the proof of Theorem~\ref{sp}, and set $\Omega_j=A_j\psi_j\Omega$. For each $j$, we define
\[
L_j:=\log(1/\delta_j),\ s_j:=\sqrt{\delta_j}, \ d_j:=L_j^{-1/2}. 
\]
Let $z\in H$, we set 
\[
\lambda_j(z):=\exp\left(-d_j\operatorname{Log}(1-is_jz)\right),
\]
where $\operatorname{Log}$ is the principal branch of the right half plane. Then $\lambda_j$ is holomorphic in $H$ and satisfies the following properties
\begin{enumerate}
    \item $|\lambda_j(z)|\leq1$ and $|\operatorname{arg}(\lambda_j(z))|\leq\pi d_j/2$,

    \item $\lambda_j(z)\to1$ uniformly on compact sets of $H$,

    \item for each $\eta>0$,
    \[
    \sup_{z\in H,|\delta_jz|\geq\eta}|\lambda_j(z)|\leq\exp\left(-\frac{\sqrt{L_j}}{2}-\frac{\log\eta}{\sqrt{L_j}}\right)\to0,
\]

    \item for each $z_j\in H$, the accumulation points of $\lambda_j(z_j)$ lie in $[0,1]$.
    
\end{enumerate}

According to the proof of Theorem \ref{sp}, we can define the holomorphic mapping $F_j:\Omega_j\rightarrow\mathbb{D}\times\mathbb{C}$ by
\[
F_j(z,w)=\left(\frac{z-i}{z+i},\lambda_j(z)w\right).
\]
Since $\lambda_j(z)$ has no zero for each $j$, then such $F_j$ is a holomorphic embedding. And $F_j$ converges to 
\[
F(z,w):=\left(\frac{z-i}{z+i},w\right)
\]
uniformly on compact sets of $H\times D_2$.
\smallskip

\noindent\textbf{Claim 1}: Let $\epsilon_j:=\sup_{(z,w)\in\Omega_j}\operatorname{dist}(\lambda_j(z)w,D_2)$, then $\epsilon_j\to0$.
\smallskip

Suppose, for a contradiction, there exists $(z_j,w_j)\in\Omega_j$ such that 
\[
\liminf{_{j\to\infty}}\operatorname{dist}(\lambda_j(z_j)w_j,D_2)\geq\epsilon_0>0.
\]
After passing to a subsequence, we may assume
\[
w_j\to w,\ \lambda_j(z_j)\to t\in[0,1].   
\]
If $t=0$, then $\lambda_j(z_j)w_j\to0\in D_2$, which contradicts our assumption. If $t>0$, we know that $\delta_jz_j\to0$ from the properties of $\lambda_j(z)$, and hence $w\in\overline{D}_2$. Since $D_2$ is convex, then $\lambda_j(z_j)w_j\to tw\in \overline{D}_2$, that contradicts our assumption again.

\smallskip

\noindent\textbf{Claim 2:} Let $\Sigma_j:=F_j(\Omega_j)$, then $\overline{\Sigma}_j$ converges to $\overline{\Sigma}:=\overline{\mathbb{D}}\times\overline{D}_2$ in the Hausdorff distance.

It is enough to prove that: for any $\epsilon>0$, there exists $k_0>0$ such that 
\[
\overline{\Sigma}_j\subset N(\overline{\Sigma},\epsilon),\ \operatorname{and}\  \overline{\Sigma}\subset N(\overline{\Sigma}_j,\epsilon)
\]
for any $j\geq k_0$, where $N(\cdot,\epsilon)$ is the $\epsilon$-neighborhood. Note that each $\overline{\Sigma}_j$ is contained in $\overline{\mathbb{D}}\times\overline{B(0,R)}$ with some constant $R>0$, independently of $j$. From Claim 1, we know that
\[
\sup_{x\in\overline{\Sigma}_j}\operatorname{dist}(x,\overline{\mathbb{D}}\times\overline{D}_2)\to0,
\]
which means that $\overline{\Sigma}_j\subset N(\overline{\Sigma},\epsilon)$.
On the other hand, we can select finite points $p_1\cdots,p_l\in\mathbb{D}\times D_2$ such that 
\[
\overline{\Sigma}\subset\bigcup_{i=1}^lB_{\epsilon/2}(p_i).
\]
Let $q_i=F^{-1}(p_i)\in H\times D_2$, then each $q_i$ is contained in $\Omega_j$ for large $j$. Since $F_j$ converges to $F$ uniformly on compact sets of $H\times D_2$, then $|F_j(q_i)-F(q_i)|<\epsilon/2$ for large $j$. Now for any $z\in\overline{\Sigma}$, there exists $p_{i_0}\in\mathbb{D}\times D_2$ such that $|z-p_{i_0}|<\epsilon/2$. Hence
\[
\operatorname{dist}(z,\overline{\Sigma}_j)\leq|z-F_j(q_{i_0})|\leq|z-p_{i_0}|+|p_{i_0}-F_j(q_{i_0})|<\epsilon,
\]
 and this means that $\overline{\Sigma}\subset N(\overline{\Sigma}_j,\epsilon)$ .

\smallskip

\noindent\textbf{Claim 3}: Let $K\subset\mathbb{D}\times D_2$ be a compact set, then $K\subset\Sigma_j$ for all large $j$. 
\smallskip

Since $K$ is a compact set of $\mathbb{D}\times D_2$, then $F^{-1}(K)$ is compact in $H\times D_2.$ Let $V$ be a neighborhood of $F^{-1}(K)$ that is compactly contained in $H\times D_2$. Note that $\Omega_j$ converges to $H\times D_2$ in $\mathbb{X}_2$, hence $V$ is compactly contained in $\Omega_j$ for large $j$. Since $F_j^{-1}$ converges to $F^{-1}$ uniformly on compact sets of $\mathbb{D}\times D_2$, then $V$ contains $F^{-1}_j(K)$ for large $j$, and hence $K\subset\Sigma_j$ for large $j$.

\medskip

Since $D_2$ is a bounded simply connected domain in $\mathbb{C}$, then there exists a sequence of bounded simply connected domains $E_k$ such that 
\[
E_{k+1}\subset E_k,\ D_2\subset\subset E_k\ \operatorname{for\ each}\ k
\]
and $E_k$ converges to $D_2$ in the sense of kernel convergence. From the Claim 2, for each $k$, there exists $j_0>0$ such that 
\[
\Sigma_j\subset \mathbb{D}\times E_k
\]
whenever $j\geq j_0$. By Riemann mapping theorem, there exists a unique biholomorphism $g_k: E_k\rightarrow\mathbb{D}$ such that $g_k(0)=0$ and $g_{k}'(0)>0$, which means that the map 
\[
f_k(z,w)=(z,g_k(w)):\mathbb{D}\times E_k\rightarrow\mathbb{D}^2
\]
is a biholomorphism. Suppose $g:D_2\rightarrow\mathbb{D}$ is the unique biholomorphism such that $g(0)=0$ and $g'(0)>0$, then, by Proposition \ref{kc}, $g_{j}$ converges to $g$ uniformly on compact subsets of $D_2$.

For each fixed $1>r>0$, there exists $k$ such that $g_k^{-1}(D_r(0))$ is compactly contained in $D_2$, and hence $D_r(0)\times g_k^{-1}(D_r(0))$ is compactly contained in $\mathbb{D}\times D_2$.  Now for such $k$, there exists some constant $j_k\geq1$ such that $\Sigma_j\subset\mathbb{D}\times E_k$ and $D_r(0)\times g_k^{-1}(D_r(0))\subset\Sigma_j$ for any $j\geq j_k$. Let $G_k=(z,g_k(w)):\mathbb{D}\times E_k\to\mathbb{D}^2$ and $f_j=G_k\circ F_j\circ A_j\circ\psi_j$, then $f_j$ is a holomorphic embedding from $\Omega$ to $\mathbb{D}^2$ with $f_j(p_j)=0$ and $D_r(0)\subset f_j(\Omega)$. This means that
\[
\liminf_{j\rightarrow\infty}T_{\Omega}(p_j)\geq r.
\]
Since $r$ is arbitrary, we obtain that
\[
\lim_{j\to\infty}T_{\Omega}(p_j)=1.
\]
The same contradiction argument as in Theorem~\ref{ls} finally yields that
\[
\lim_{z\rightarrow p}T_{\Omega}(z)=1,
\]
and this completes the proof.
\end{proof}

\smallskip

Now we prove that the boundary behavior of $T_{\Omega}$ implies the Levi flatness of boundary points. Before proving Theorem \ref{elf}, we establish the following gap lemma, from which Theorem \ref{elf} follows directly.
\begin{lemma}\label{gap}
    For any $n\geq2$, there exists a constant $\epsilon=\epsilon(n)>0$ such that : if $\Omega$ is a convex domain in $\mathbb{X}_n$, then
    \[
    S_{\Omega}(z)+T_{\Omega}(z)<2-\epsilon
    \]
    holds for every $z\in\Omega$.
\end{lemma}

\begin{proof}
   For fixed $n\geq2$. Suppose, for a contradiction, that for each $\epsilon_k>0$ there exists a convex domain $\Omega_k\in\mathbb{X}_n$ such that 
   \[
   2-\epsilon_k\leq S_{\Omega_k}(z_k)+T_{\Omega_k}(z_k)
   \]
   holds for some $z_k\in\Omega_k$. We may assume that $\epsilon_k\rightarrow0$. For each $(\Omega_k,z_k)$, by Theorem \ref{com}, there exists $A_k\in\operatorname{Aff}(\mathbb{C}^n)$ such that $A_k(\Omega_k,z_k)\in\mathbb{K}_n$. After passing to a subsequence, there exists $(\Omega_{\infty},z_{\infty})\in\mathbb{X}_n$ such that $A_k(\Omega_k,z_k)$ converges to $(\Omega_{\infty},z_{\infty})$ in $\mathbb{X}_{n,0}$. Proposition \ref{uc} gives 
   \[
   \limsup_{k\rightarrow\infty}(2-\epsilon_k)\leq\limsup_{k\rightarrow\infty} (S_{\Omega_k}(z_k)+T_{\Omega_k}(z_k))\leq S_{\Omega_{\infty}}(z_{\infty})+T_{\Omega_{\infty}}(z_{\infty})\leq2
   \]
   and this implies
   \[
   S_{\Omega_{\infty}}(z_{\infty})=T_{\Omega_{\infty}}(z_{\infty})=1.
   \]
   Hence $\Omega_{\infty}$ is biholomorphic to both $\mathbb{B}^n$ and $\mathbb{D}^n$, and this is a contradiction.
\end{proof}

\noindent$Proof\ of\ Theorem\ \ref{elf}$: Since $\Omega'$ is biholomorphic to a bounded convex domain, then $\Omega'$ is pseudoconvex. Suppose there exists a smooth point $p\in\partial\Omega'$ such that $\partial\Omega'$ is not Levi flat at $p$, then $p$ is a strongly pseudoconvex point. Choosing $p_k\in\Omega'$ that converges to $p$, then 
\[
\lim_{k\rightarrow\infty}S_{\Omega'}(p_k)=1.
\]
Now suppose $\Omega$ and $\Omega'$ are biholomorphically equivalent via a biholomorphic map $f:\Omega\rightarrow\Omega'$, then 
\begin{align}\label{lim}
\lim_{k\rightarrow\infty}S_{\Omega}(f^{-1}(p_k))=1.
\end{align}
After passing to a subsequence, we assume $f^{-1}(p_k)$ converges to a boundary point $q\in\partial\Omega$. By Theorem \ref{main}, we know that 
\[
\lim_{k\rightarrow\infty}T_{\Omega}(f^{-1}(p_k))=1,
\]
this contradicts Lemma \ref{gap} in view of (\ref{lim}).
\smallskip

\noindent$Proof\ of\ Corollary\ \ref{flatness}$: Suppose, for a contradiction, that for any neighborhood $V$ of $p$, there exists $z\in\partial\Omega\cap V$ such that $\partial\Omega$ is not Levi flat at $z$, that is $z$ is a strongly pseudoconvex point. Then we obtain that
\[
\lim_{w\rightarrow z}S_{\Omega}(w)=1.
\]
Therefore, we can select $z_n\in\partial\Omega$ that converges $p$ and each $z_n$ is a strongly pseudoconvex point. Hence, for any $\epsilon>0$, there exists $q_n\in\Omega$ that converges to $p$ 
\[
S_{\Omega}(q_n)>1-\epsilon.
\]
From the assumption of Corollary \ref{flatness}, the following
\[
T_{\Omega}(q_n)>1-\epsilon
\]
holds for large $n$. This contradicts Lemma \ref{gap}, and we complete the proof.$\hfill\square$
\medskip

At the end of this section, we now prove Theorem \ref{fv}. By Remark \ref{us}, we recall the following two lemmas.

\begin{lemmaNoParens}[{\cite[Theorem 2.1]{finite}}]\label{ve}
    If $\Omega$ satisfies the hypotheses of Theorem \ref{fv}, then there exists $C>1,\epsilon>0$ such that

    (1) $g_{B},g_{KE}$ and $k_{\Omega}$ are all $C$-bi-Lipschitz,

    (2) if $\operatorname{Vol}$ denotes the Bergman volume, the K\"ahler-Einstein volume, the Kobayashi-Eisenman volume, then
    \begin{align}
    \frac{1}{C}r^{4}\leq\operatorname{Vol}\left(\left\{z\in\Omega:B_{\Omega}(z_0,z)\leq r\right\}\right)\leq Cr^4   
    \end{align}
    for all $r\in[0,\epsilon]$ and $z_0\in\Omega$. 
\end{lemmaNoParens}

\begin{lemmaNoParens}[{\cite[Theorem 12.2]{fr}}]\label{fix}
    Let $\Omega\subset\mathbb{C}^d$ be a bounded convex domain and $K\leq\operatorname{Aut}(\Omega)$ be a compact group. Then there exists a    point $z\in\Omega$ such that $\gamma(z)=z$ for any $\gamma\in K$.
\end{lemmaNoParens}

\noindent$Proof\ of\ Theorem\ \ref{fv}$: In this proof, $B_{\Omega}$ is denoted the distance induced by the Bergman metric $g_{B}$.   Suppose $A\subset\Omega$, we let $\widetilde{\operatorname{Vol}}(A)$ denote the volume relative to the associated measure on $\Omega$. If $\pi:\Omega\rightarrow\Gamma\setminus\Omega$ is the natural covering map and $\pi|_{A}$ is injective, then
   \[
   \widetilde{\operatorname{Vol}}(A)=\operatorname{Vol}(\pi(A)).
   \]
   
   Let $p\in\partial\Omega$ be a smooth point, if $p$ is a strongly pseudoconvex point, then $\Omega$ is biholomorphic to $\mathbb{B}^2$ by a similar proof that in \cite{finite}. And this contradicts \cite[Theorem 1.1]{pinchuk}. Hence we know that every smooth boundary point of $\partial\Omega$ is Levi flat. Since the proof of Theorem \ref{main} relies on the Levi flatness along each smooth face, then we know that $\lim_{z\to p}T_{\Omega}(z)=1$ for such $\Omega$ and each $p\in\partial\Omega$.

   Let $p_n\in\Omega$ converge to a singular boundary point $p\in\partial\Omega$. Then Lemma \ref{ve} implies that there exists $C_1>0,\epsilon_1>0$ such that
   \begin{align}
   \widetilde{\operatorname{Vol}}\left(\left\{z\in\Omega:B_{\Omega}(p_n,z)<r\right\}\right)\geq C_1r^4    
   \end{align}
   for all $n\geq 1$ and $r\in[0,\epsilon_1]$.

   For each $n\geq1$, we let 
   \[
   \delta_n:=\min_{\gamma\in\Gamma\setminus\{1\}}B_{\Omega}(p_n,\gamma p_n).
   \]
   Then the quotient map $\pi:\Omega\rightarrow\Gamma\setminus\Omega$ restricts an embedding on
   \[
   B_n:=\left\{z\in\Omega:B_{\Omega}(z,p_n)<\delta_n/2\right\},
   \]
  and we obtain that
    \[
    \operatorname{Vol}\left(\pi(B_n)\right)=\widetilde{\operatorname{Vol}}(B_n)\geq C_1\min\{\epsilon_1^4,(\delta_n/2)^4\}.
    \]
    After passing to a subsequence, we may assume that
    \[
    \lim_{n\rightarrow\infty}\delta_n=\delta\geq0.
    \]

    \textbf{Case 1}: $\delta\neq0$. Let $r=\min\{\epsilon_1,\delta/4\}$ and let 
    \[
    B_n':=\{z\in\Omega:B_{\Omega}(z,p_n)<r\}.
    \]
    Then there exists $N\geq 1$ such that $B'_{n}\subset B_n$ and  
    \[   \operatorname{Vol}\left(\pi(B'_n)\right)=\widetilde{\operatorname{Vol}}(B'_n)\geq C_1r^4
    \]
    for all $n\geq N$.   

    Suppose $M:=\Gamma\setminus\Omega$ and $d$ is the distance on $M$ making $\pi:(\Omega,B_{\Omega})\rightarrow(M,d)$ is a local isometry.
    \smallskip
    
   \noindent \textbf{Claim}:$\{\pi(p_n):n\geq1\}$ is relatively compact in $M$.
    \smallskip

    \noindent $Proof\ of\ claim$: Suppose not, then there exists $n_i$ such that
    \[
    n_1<n_2<\cdots
    \]
    and
    \[
    \min_{1\leq k<j}d\left(\pi(p_{n_k}),\pi(p_{n_j})\right)>2r
    \]
    for each $j\geq2$.
    Thus the sets
    \[
    \pi(B'_{n_1}),\pi(B'_{n_2}),\cdots
    \]
    are pairwise disjoint since
    \[
    \pi(B'_{n_j})\subset\{z\in M:d(z,\pi(p_{n_j}))<r\}.
    \]
    Hence we obtain
    \[
    \operatorname{Vol}(M)\geq\sum_{j\geq1}^{\infty}\operatorname{Vol}(\pi(B_{n_j}'))\geq\sum_{j\geq1}^{\infty}C_1r^4=\infty ,   
    \]
    which is a contradiction. 
\smallskip

    From the claim, we know that for each $n\geq1$, there exists $\gamma_n\in\Gamma$ such that the set $\{\gamma_np_n:n\geq1\}$ is relatively compact in $\Omega$. After passing to a subsequence, we may assume $\gamma_np_n\rightarrow q\in\Omega$, and $\gamma_n^{-1}$ converges to a holomorphic map $g:\Omega\rightarrow\overline{\Omega}$. Moreover, 
    \[
    p=\lim_{n\rightarrow\infty}p_n=\lim_{n\rightarrow\infty}\gamma_n^{-1}(\gamma_np_n)=g(q)=\lim_{n\rightarrow\infty}\gamma_{n}^{-1}q.
    \]
    This means that $T_{\Omega}(q)=1$ from Theorem \ref{main} and $\Omega$ is biholomorphic to $\mathbb{D}^2$.
\medskip

    \noindent\textbf{Case 2}: $\delta=0$. For each $n$ we select $\gamma_n\in\Gamma$ such that
    \[
    B_{\Omega}(p_n,\gamma_np_n)=\delta_n.
    \]
\smallskip

    \noindent{}\textbf{Case 2(a)}: If $\{\gamma_1,\gamma_2,\cdots\}$ is infinite. After passing to a subsequence, we may assume that $\gamma_n\rightarrow\infty$ since $\Gamma$ is discrete. For fixed $z\in\Omega$, since $\Gamma$ acts properly on $\Omega$, then we may assume $\gamma_n^{-1}z\rightarrow\xi\in\partial\Omega$. Hence Theorem \ref{main} implies that $T_{\Omega}(z)=1$ and $\Omega$ is biholomorphic to $\mathbb{D}^2$. 
\smallskip

    \noindent\textbf{Case 2(b)}: If $\{\gamma_1,\gamma_2,\cdots\}$ is finite, then we may assume $\gamma_n=\gamma$ for large $n$.
\smallskip

    For fixed $z\in\Omega$, if $\{\gamma^n(z):n\geq1\}$ is relatively compact in $\Omega$, then $\gamma$ has a fixed point in $\Omega$ since Lemma \ref{fix}. But $\Gamma$ acts freely on $\Omega$, hence $\{\gamma^n(z):n\geq1\}$ is unbounded in $\Omega$. After passing to a subsequence, we may assume $\gamma^n(z)\rightarrow z_0\in\partial\Omega$. By a similar argument in case 2(a), we know that $\Omega$ is biholomorphic to $\mathbb{D}^2$. This completes the proof.
$\hfill\square$

\section{Biholomorphism-type of $\mathcal{T}_g$}\label{sec5}
 In this section, we prove Theorem \ref{te}, Theorem \ref{peak} and Theorem \ref{lcb}. Our proofs are mainly based on the techniques in \cite{pinchuk,tei}. Before proving the theorems, we recall some properties of $\mathcal{T}_g$.

 Let $S$ be a compact surface with genus $g\geq2$, and $\mathcal{T}_g$ be the \T\ space associated to $S$. The following facts are well known:

 \begin{enumerate}
     \item $\mathcal{T}_g$ cannot be biholomorphic to a bounded convex domain in $\mathbb{C}^{3g-3}$.

     \item The automorphism group of $\mathcal{T}_g$ is discrete.
 \end{enumerate}

Now we start the proof of our main theorems about the biholomorphism-type of \T\ space. We first prove the following result.

\begin{proposition}\label{prop:peakorbit}
Let $\Omega\subset\mathbb{C}^{3g-3}$ be a bounded domain that is biholomorphic to some $\mathcal T_g$ with $g\geq2$. If $p\in\partial\Omega$ is a local holomorphic peak point, then there exists some $q\in\Omega$ and $f_j\in\operatorname{Aut}(\Omega)$ such that $f_j(q)\to p$ as $j\to\infty$.
\end{proposition}
\begin{proof}
Since $p\in\partial\Omega$ is a local holomorphic peak point, we can take a holomorphic peak function $g$ on a neighborhood $U$ of $p$ such that $g(p)=1$ and $|g|<1$ on
$(\overline\Omega\cap U)\setminus\{p\}$. Now let $h=1-\operatorname{Re} g\geq0$, then $h$ is pluriharmonic in $U$. For fixed small $r>0$ with $\overline{B_r(p)}\subset U$, we let
\[
 m_r:=\min\left\{h(z):z\in\overline\Omega,\ r/2\leq|z-p|\leq r\right\}>0
 \]
 and
 \[
 a_r=m_r/2.
\]
With the function $h$ above, we define 
\[
 u_r(z):=
 \begin{cases}
 \max\left\{-h(z),-a_r\right\},&|z-p|<r,\\
 -a_r,&|z-p|\geq r
 \end{cases}                                               
\]
on whole $\overline\Omega$. By our definition of $m_r$ and $u_r$, we know that $h(z)\geq m_r=2a_r>0$ for all $|z-p|= r$, hence $u_r(z)=-a_r$ for all $|z-p|=r$. Consequently, $u_r$ is continuous on $\overline\Omega$ and plurisubharmonic
on $\Omega$ with $-a_r\leq u_r\leq0$. Moreover, we have
\begin{align}\label{subset}
 \{u_r>-a_r\}\subset B_{r/2}(p)            
\end{align}
and $u_r(p)=0$. 

Now we choose $q_r\in\Omega$ close enough to $p$ such that that $u_r(q_r)>-a_r$. By \cite[Proposition~2.2]{tei}, there is a Teichm\"uller disc
$\phi:\mathbb{D}\to\Omega$ with $\phi(0)=q_r$ such that a full measure set of radial rays has dense projection to moduli space. Such rays return arbitrarily
late to any fixed sufficiently small relatively compact open set in moduli space. Indeed, a finite initial geodesic segment has compact image
with empty interior. If a tail missed an open set, then a smaller open subset
disjoint from that initial image would be missed by the entire ray. This contradicts that the projection of the 
geodesic ray is dense in the moduli space.
Since $\Omega$ is bounded, then Theorem 3.6 in \cite{tei} implies that $\phi$ has Euclidean radial limits
$\widehat\phi(e^{i\theta})\in\overline\Omega$ for almost every $\theta$. The plurisubharmonicity of $u_r$ means that $u_r\circ\phi$ is a bounded subharmonic function.
After applying the submean inequality to $u_r\circ\phi$, we obtain that
\[
u_r(q_r)\leq\frac{1}{2\pi}\int_0^{2\pi}u_r(\phi(se^{i\theta}))d\theta
\]
for any $s\in(0,1)$. Letting $s\to1$, since $u_r\circ\phi$ is continuous, then dominated convergence theorem implies the following
\begin{align}\label{p3}
u_r(q_r)\leq\frac1{2\pi}\int_0^{2\pi}
       u_r(\widehat\phi(e^{i\theta}))\,d\theta.                 \end{align}
Since $u_r(q_r)>-a_r$, it follows that the measurable set
\[
E_r=\{e^{i\theta}:u_r(\widehat\phi(e^{i\theta}))>-a_r\}
\]
has positive measure. Indeed, if the set $E_r$ has zero measure, then the right side of (\ref{p3}) will less or equal than $-a_r$. This contradicts $u_r(q_r)>-a_r$. By (\ref{subset}), each $\widehat\phi(e^{i\theta})$ lies in $B_{r/2}(p)$ for each $\theta\in E_r$. After applying a similar dicussion in the proofs of Proposition 3.4 and Lemma 3.9 in \cite{tei}, we know that there exists a positive measure subset of $E_{1,r}\subset\partial\mathbb{D}$ such that for each $\theta\in E_{1,r}$ the geodesic ray $\gamma_{\theta}=\{\phi(re^{i\theta}):0<r<1\}$ eventually lies in $B_r(p)$.

Now fix a compact subset $K$ of moduli space $\mathcal{M}_g$. For $r_n\to0$, Proposition 3.10 in \cite{tei} implies that there exists $p_n\in\Omega$ converges to $p$ in the Euclidean metric and $p_n$ projects to the compact set $K\subset\mathcal{M}_g$. 
Since $\operatorname{MCG}(S)=\operatorname{Aut}(\Omega)$, then there exists $\gamma_n\in\operatorname{Aut}(\Omega)$ such that $x_n=\gamma_n(p_n)$ lie in a fixed compact lift $K_0\subset\Omega$ of $K$. After passing to a subsequence, we may assume that $x_n\to x\in\Omega$ and $\gamma_n^{-1}(x_n)\to p$. Note that there is $C>0$ with
$|z-w|\leq C K_\Omega(z,w)$ on $\Omega$, where $K_{\Omega}$ is the Kobayashi distance on $\Omega$. Therefore,
\[
 |\gamma_n^{-1}(x)-\gamma_n^{-1}(x_n)|
 \leq C K_\Omega(\gamma_n^{-1}(x_n),\gamma_n^{-1}(x))
 =C K_\Omega(x,x_n)\longrightarrow0.
\]
This means that
\[
|\gamma_n^{-1}(x)-p|\leq|\gamma_n^{-1}(x_n)-\gamma_n^{-1}(x)|+|\gamma_n^{-1}(x_n)-p|\to0
\]
Thus $\gamma_n^{-1}(x)\to p$ and this completes the proof.
\end{proof}

The following lemma states that a boundary point that is a good corner is a local holomorphic peak point.
\begin{lemma}\label{lem:goodpeak}
Every point satisfying (C.1)--(C.5) in Definition \ref{good corner} is a local holomorphic peak point.
\end{lemma}
\begin{proof}
Suppose $(\Omega,p,\rho_1,\cdots,\rho_k)$ satisfies the conditions (C.1)--(C.5) in Definition \ref{good corner}. Let $\rho=\sum_{i=1}^k\rho_i+A\sum_{i=1}^k\rho_i^2$. After shrinking the neighborhood $U$ appeared in Definition \ref{good corner}, we may assume
that $1+A\rho_i>0$. Since $\rho_i\leq0$ on $\overline\Omega$, then
$\rho\leq0$ on $U\cap\overline{\Omega}$  and $\rho(p)=0$.
Let $\xi=z-p$ and we define the holomorphic Levi polynomial
\[
 L(z)=\sum_{i=1}^n\frac{\partial\rho}{\partial z_i}(p)\xi_i
       +\frac12\sum_{i,j=1}^n
                 \frac{\partial^2\rho}{\partial z_i\partial z_j}(p)\xi_i\xi_j.
\]
Note that the $C^2$ Taylor expansion of $\rho$ at $p$ is
\[
 \rho(z)=2\operatorname{Re} L(z)+\sum_{i,j=1}^n\frac{\partial^2\rho}{\partial z_i{\partial}\overline{z}_j}
       (p)\xi_i\bar\xi_j
          +o(|\xi|^2).
\]
Since $\rho$ is strictly plurisubharmonic on $U$, then
\[
2\operatorname{Re} L(z)\leq-c|z-p|^2
\]
with some $c>0$ on $\overline\Omega\cap U$ after shrinking $U$ if necessary. Thus $e^{L(z)}$ is the required holomorphic peak function.
\end{proof}

Now we prove Theorem \ref{te}, we restate it as follows for convenience.
\begin{thm}\label{lp}
    Any Teichm\"{u}ller space $\mathcal{T}_g$ with $g\geq2$ cannot be biholomorphic to a bounded domain with piecewise $C^2$-smooth boundary.
\end{thm}
\begin{proof}
Suppose a bounded domain $\Omega$ with piecewise $C^2$ boundary that is
biholomorphic to $\mathcal T_g$ with $g\geq2$.
Then the domain is pseudoconvex, since Teichm\"uller space is Stein.  By Pinchuk's existence result \cite[Lemma 1.2-1.3]{pinchuk}, there exists a boundary point $p\in\partial\Omega$ with index $r\geq1$ satisfies condition (C.1)--(C.5) in Definition \ref{good corner}. Therefore, Lemma~\ref{lem:goodpeak} and
Proposition~\ref{prop:peakorbit} state that there exists $q\in\Omega$ and $f_j\in\operatorname{Aut}(\Omega)$ such that $f_j(q)\to p$.
Since such $p$ satisfies the condition in \cite[Theorem I]{K86}, then we conclude that $\Omega$ is biholomorphic to the domain
    \[
    W:=\left\{w\in\mathbb{C}^n:\operatorname{Re}w_i+H_i(w')<0,\ i=1,\cdots,r\right\},
    \]
   where $w'=(w_{r+1},\cdots,w_n)$ and $H_i$ are nonnegative quadratic Hermitian forms such that their sum is strictly positive. For each $t\in\mathbb{R}$, the map
   \[
   (w_1,w_2,\cdots,w_n)\to(w_1+it,w_2,\cdots,w_n)
   \]
   is an automorphism of $W$. And this contradicts the fact that the automorphism group of $\mathcal{T}_{g}$ is discrete.
\end{proof}

\noindent$Proof\ of\ Theorem\ \ref{peak}$: Suppose $\Omega$ is a bounded domain that is locally convexifiable at a local holomorphic peak point $p\in\partial\Omega$ and $\Omega$ is biholomorphic to some $\mathcal{T}_g$ with $g\geq2$. Then Proposition \ref{prop:peakorbit} implies that such $p$ is an orbit accumulation of $\operatorname{Aut}(\Omega)$. That is there exists $f_j\in\operatorname{Aut}(\Omega)$ and $q\in\Omega$ such that $f_j(q)\to p$. After passing to a subsequence, we assume $f_j\to f$ uniformly on compact sets of $\Omega$ with $f(q)=p$. Since $p$ is a local holomorphic peak point, there does not exist a non-constant affine disk in $\partial\Omega$ through $p$. Hence $f(z)\equiv p$ for all $z\in \Omega$. Now suppose $(U,\Phi)$ satisfies Definition \ref{convexifiable}, then for each compact set $K\subset\Omega$, $f_j(K)\subset A:=U\cap\Omega$ for large $j$. Let $\Omega_j:=f_j^{-1}(A)\subset\Omega$, then any compact set $K\subset\Omega$ is contained in $\Omega_j$ for large $j$. Since $\Phi(A)$ is convex, then $K_{\Phi(A)}(z,w)=C_{\Phi(A)}(z,w)$ for any $z,w\in\Phi(A)$ and hence $K_{\Omega_j}(x,y)=C_{\Omega_j}(x,y)$ for any $x,y\in\Omega_j$. Here $K_{\Omega_j},C_{\Omega_j}$ are the  Kobayashi distance and Carath\'eodory distance on $\Omega_j$. Now for each $x,y\in\Omega$, we choose $h_j:\Omega_j\to\mathbb{D}$ such that $h_j(x)=0$ and
\begin{align}\label{inq}
\rho_{\mathbb{D}}(0,h_j(y))\geq C_{\Omega_j}(x,y)-\frac{1}{j}=K_{\Omega_j}(x,y)-\frac{1}{j}\geq K_{\Omega}(x,y)-\frac{1}{j}.
\end{align}
By Montel's theorem, after passing to a subsequence, $h_j$ converges to $h:\Omega\to\overline{\mathbb{D}}$ uniformly on compact sets of $\Omega$. Since $h(x)=0$, then $h(\Omega)\subset\mathbb{D}$, and
\[
C_{\Omega}(x,y)\geq\rho_{\mathbb{D}}(h(x),h(y))\geq K_{\Omega}(x,y)
\]
from (\ref{inq}). This implies that $C_{\Omega}(x,y)=K_{\Omega}(x,y)$ for any $x,y\in\Omega$, which contradicts that the Kobayashi distance and Carath\'eodory distance do not coincide in $\mathcal{T}_g$ with $g\geq2$. And the proof of Theorem \ref{peak} is completed.
$\hfill\square$

\medskip

\noindent$Proof\ of\ Theorem\ \ref{lcb}$: Suppose $\Omega$ is biholomorphic to some $\mathcal{T}_g$ with $g\geq2$ and $p\in\overline{\Omega}$ satisfies
\[
|p|:=\max\left\{|w|:w\in\overline{\Omega}\right\},
\]
then we show that such $p$ is a local holomorphic peak point. Note that such $p$ must be a boundary point. Now consider
\[
H:=\left\{w\in\mathbb{C}^{3g-3}:\operatorname{Re}\langle w,p\rangle=R^2\right\}
\]
with $R=|p|$. For any $w\in\overline{\Omega}$, we have
\[
\operatorname{Re}\langle w,p\rangle\leq|w||p|\leq R^2
\]
and the equality holds only if $w=p$. Hence
\[
H\cap\overline{\Omega}=\{p\}.
\]

Now let $L(z):=\langle z,p\rangle$ and $f(z):=\exp(L(z)-R^2)$, then $f$ is holomorphic and hence $f$ is a peak function of $p$ from the above discussion.  After combining Theorem \ref{peak}, this completes the proof of Theorem \ref{lcb}.

\vspace*{3mm}
\noindent {\bf Funding}.
Xingsi Pu was supported by the National Natural Science Foundation of China (Grant No. 12501091), the Science and Technology Research Program of Chongqing Municipal Education Commission (Grant No. KJQN202501110), Chongqing Natural Science Foundation Innovation and Development Joint Fund (CSTB2025NSCQ-LZX0059), Scientific Research Foundation of the Chongqing University of Technology (No. 2024ZDZ027). Both authors were supported by National Key R\&D Program of China (Grant No. 2021YFA1003100).

\bibliography{ref}
\bibliographystyle{plain}{}
\end{document}